\documentclass[11pt]{amsart} 

\usepackage{setspace} \onehalfspace
\usepackage[active]{srcltx}
\usepackage[all]{xy} 
\usepackage{amsmath} 
\usepackage{amssymb}
\usepackage{color}
\usepackage{xcolor}
\usepackage{amsthm} \usepackage[pdfauthor={Matei Toma},
pdftitle={Bounded sets of sheaves on K\"ahler manifolds II},
pdfkeywords={Douady space, coherent sheaves, complex space}, pdfview={Fit}]{hyperref}

\newenvironment{pf}{\begin{proof}}{\end{proof}}
\newtheorem{Th}{Theorem}[section] 

\newtheorem{Lem}[Th]{Lemma}
\newtheorem{Cor}[Th]{Corollary} 
\newtheorem{Prop}[Th]{Proposition}

\newtheorem{Def}[Th]{Definition}

\newtheorem{Rem}[Th]{Remark}

\newcommand{\cA}{\mathcal{A}} 
\def\cC{\mathcal{C}} 

\def\cE{\mathcal E} 
\def\cF{\mathcal F}

\newcommand{\cI}{\mathcal{I}}

\newcommand{\cO}{\mathcal{O}}

\def\gE{{\mathfrak E}}
\def\gF{{\mathfrak F}}

\def\cringle{\mathaccent23}

\def\cI{\mathcal I }

\newcommand{\C}{\mathbb{C}} 
\newcommand{\N}{\mathbb{N}}
\renewcommand{\P}{\mathbb{P}} 
\newcommand{\Q}{\mathbb{Q}}
\newcommand{\R}{\mathbb{R}}
\newcommand{\Z}{\mathbb{Z}}

\def\ch{\mbox{ch}}
\def\Coh{\mbox{Coh}}
\def\Coker{\mbox{Coker}}
\def\cycle{\mbox{cycle}} 
\def\d{\mbox{d}}
 \def\deg{\mbox{deg}}
 
\def\Im{\mbox{Im}}

\def\Hom{\mbox{Hom}} 
\def\Ker{\mbox{Ker}}
 \def\max{\mbox{max}} 
 \def\min{\mbox{min}} 
 
\def\Pic{\mbox{Pic}}

\def\rank{\mbox{rank}}

\def\Spec{\mbox{Spec}}

\def\Supp{\mbox{Supp}}

\def\Tors{\mbox{Tors}}

\def\Td{\mbox{Td}}
\def\vol{\mbox{vol}}

\title{Bounded sets of sheaves on relative analytic spaces}

\author[Matei Toma]{Matei Toma}
\address{Matei Toma, 
Universit\'e de Lorraine, CNRS, IECL, F-54000 Nancy, France}

\email{Matei.Toma@univ-lorraine.fr}
\urladdr{\href{https://iecl.univ-lorraine.fr/membre-iecl/toma-matei/}{https://iecl.univ-lorraine.fr/membre-iecl/toma-matei/}}

 \date{\today}
\thanks{ AMS
  Classification (2000): 32J99; secondary: 14C05.}

\begin{document}

\begin{abstract}
We extend previous results on boundedness of sets of coherent sheaves on a compact K\"ahler manifold to the relative and not necessarily smooth case. This enlarged context allows us to prove properness properties of the relative Douady space as well as results related to semistability of sheaves such as the existence of relative Harder-Narasimhan filtrations.
\end{abstract}
\maketitle

\noindent

\section{Introduction}

In his paper \cite{Grothendieck-theHilbertScheme} constructing the Hilbert scheme Grothendieck introduced the notion of a bounded set $\gE$ of isomorphism classes of  coherent sheaves  on the fibers of $X/S$, where $X$ is a scheme of finite type over a noetherian scheme $S$. 
Roughly expressed this means that the elements of $\gE$ are among the fibres of an algebraic family of coherent sheaves in the fibers of $X_{S'}/S'$ after base change to a scheme $S'$ of finite type over $S$. 
Thus in the absolute case, i.e. when $S=\Spec \, k$,  this boils down to the existence  of an algebraic family of coherent sheaves  parameterized by a  scheme of finite type over $k$ whose set of  isomorphism classes includes  $\gE$.  
When $X$ is projective over $S$ with relatively very ample sheaf $\cO(1)$ Grothendieck further gave a criterion \cite[Th\'eor\`eme 2.1]{Grothendieck-theHilbertScheme} for a set $\gE$ of  isomorphism classes of coherent sheaves on the fibers of $X/S$ to be bounded saying that this is the case if and only if all elements $E\in\gE$ appear as quotients of a uniform sheaf of the form $\cO(-n)^{\oplus N}$ and their Hilbert polynomials range within a finite set of polynomials. This criterion is essential in proving that the connected components of the Hilbert scheme of a projective scheme are projective but also in the theory of moduli spaces of semi-stable sheaves.

This paper deals with boundedness for   
sets $\gE$  of isomorphism classes of analytic coherent sheaves over relative analytic spaces $X/S$. One may rewrite almost literally  Grothendieck's definition of boundedness in this new set-up, but this will find little use in applications. What is important in our analytic context is the  aspect of relative compactness hidden behind that definition. We will therefore say that a set $\gE$  of isomorphism classes of analytic coherent sheaves on the fibers of a relative analytic space $X/S$ is bounded if it can be viewed as a relatively compact subset of a suitable analytic parameter space; see 
the precise Definition \ref{def:bounded-relative-case1}. Our main result, Theorem \ref{theorem:principala}, then gives an analogue of Grothendieck's boundedness criterion \cite[Th\'eor\`eme 2.1]{Grothendieck-theHilbertScheme} in this set-up. For its formulation we replace the Hilbert polynomials by collections of degree functions, Definitions \ref{def:degrees} and \ref{def:relative-degrees}. We apply the boundedness criterion to properness questions on the Douady space as well as to problems arising in the study of semistable sheaves, such as the existence of relative Harder-Narasimhan filtrations and openness of semistability. 
Note that in a previous paper 
\cite{TomaLimitareaI} we proved properness of the connected components of the Douady space albeit in an absolute smooth  K\"ahler context. To obtain applications to moduli spaces of semistable sheaves however, boundedness results in the relative, not necessarily smooth case were needed and this is precisely what the present paper gives. 

An important device which allowed us to prove the main result of this paper as well as \cite[Theorem 4.1]{TomaLimitareaI} was the passage from bounding quotient sheaves of a fixed coherent sheaf to bounding volumes of their projectivizations and applying Bishop's Theorem to the corresponding cycle spaces. However, unlike in \cite{TomaLimitareaI}, where the volume calculations were done globally by intersecting appropriate cohomology classes on compact K\"ahler manifolds, we perform here finer local volume computations which allow the needed flexibility. Another new feature of this paper is the use of degree systems which do not necessarily come from a fixed K\"ahler polarization. The consideration of more general degree systems is natural when working for instance with mobile classes on projective manifolds as in \cite{CampanaPaun}, \cite{CampanaCaoPaun}, \cite{GrebToma}, or \cite{GRT1}, and imposes itself in the treatment of stability of coherent sheaves over compact smooth complex non-K\"ahler surfaces, cf. \cite{LuTe}, \cite{TomaCriteria}.

The paper is organized as follows. We start by introducing degree systems for relative analytic spaces and in particular those degree systems arising from integration of polarizing differential forms along homology Todd classes of coherent sheaves. In Sections \ref{section:volume} and \ref{section:quotient-sheaves} two  ingredients of the proof of our main result are presented in detail: the local volume computations in  Section \ref{section:volume}, which roughly ensure that in order to control a true volume of a projectivized bundle a bound on a certain pseudovolume suffices, and in Section \ref{section:quotient-sheaves} the description of the passage from quotient sheaves with bounded degrees to cycles in projectivized bundles with bounded volume and back. Follow the presentation of the boundedness notion and of the main boundedness criteria in Section \ref{sec:boundedness}. In the final section applications are given to properness properties of relative Douady spaces and to the existence of relative Harder-Narasimhan filtrations.


\section{Preliminaries}\label{section:Preliminaries}
\subsection{K\"ahler morphisms}
We shall review the definition of a K\"ahler metric on a possibly singular complex analytic space $X$ in the absolute and in the relative case and shall then introduce some weaker versions of such metrics.

We start by recalling the definition of differential forms on complex spaces which we shall use. For this definition differential forms are induced from local embeddings of $X$ in a domain $V$ of $\C^n$ and it is then checked that they are well defined.

\begin{Def} (Fujiki \cite[1.1]{FujikiClosedness}) Suppose that $X$ is embedded in a domain $V$ of $\C^n$ and let $\cI_X$ be its ideal sheaf. If $\cA_V^m$ denotes the sheaf of  differential $m$-forms on $V$, one defines the {\em sheaves of differential $m$-forms} by
$$\cA_X=\cA^0_X:=\cA_V^0/(\cI_X+\bar\cI_X)\cA^0_V$$ 
and for $m>0$
$$\cA_X^m:=\cA^m_V/((\cI_X+\bar\cI_X)\cA^m_V+\d(\cI_X+\bar\cI_X)\cA^{m-1}_V).$$
\end{Def}

For further properties of these sheaves we refer the reader to  \cite[1.1]{Var89}, \cite[III.2.4]{BarletMagnusson}. We just mention the fact that type decomposition as well as exterior differentiation and the operators $\partial$ and $\bar\partial$ are well defined.

\begin{Def}
The subsheaves of $\cA_X$ of functions induced locally from smooth strongly plurisubharmonic (respectively from pluriharmonic) functions on $V$ will be denoted by $SPSH_X$ (resp. by $PH_X$).  
If $f:X\to S$ is a morphism of complex spaces we define the sheaves $SPSH_f$ and $PH_f$ as the subsheaves of functions of $\cA_X$ which are strongly plurisubharmonic (respectively pluriharmonic) on the fibers of $f$. We will call such functions {\em plurisubharmonic (resp. pluriharmonic) relatively with respect to $f$.}
\end{Def}

\begin{Def}\label{def:kaehler-morphism} (\cite[Definition 2.1]{Fujiki-Space-of-Divisors},\cite[5.1]{Bin83}) A morphism $f:X\to S$  of complex spaces is said to be {\em K\"ahler} if there exist an open cover $(U_i)_{i\in I}$ of $X$ and sections $\phi_i\in SPSH_f(U_i)$ such that $(\phi_i-\phi_j)|_{U_i\cap U_j}\in PH(U_i\cap U_j)$ for all $i,j\in I$. When the differences $\phi_i-\phi_j$ are only required to be in $PH_f(U_i\cap U_j)$ we say that $f$ is {\em weakly K\"ahler}. Such a data $(U_i,\phi_i)_{i\in I}$ is called a {K\"ahler metric} (resp. a {\em weakly K\"ahler metric}) for $f$. When $S$ is a point the two notions coincide and we say then that $X$ is a {\em K\"ahler space}.
\end{Def}

If $(U_i,\phi_i)_{i\in I}$ is a K\"ahler metric for $f:X\to S$ then the $(1,1)$-forms $i\partial\bar\partial\phi_j$ glue together to give a closed $(1,1)$-form $\omega$ on $X$ which we call {\em the
K\"ahler form of the metric}.  
If $(U_i,\phi_i)_{i\in I}$ is only weakly K\"ahler, we only get a "relative" $(1,1)$-form, cf. \cite{Bin83}.
 Note that in both cases the restriction of $\omega$ to the fibers of $f$ is {\em strictly positive}, in the sense that for local embeddings of the fibers in open domains of number spaces the form appears as the resriction of a strictly positive form. We shall use the analogous definition of positivity (in Lelong's sense) more generally for $(p,p)$-forms and in particular for the powers of $\omega$, cf. \cite{BarletMagnusson}.


\subsection{Homology classes and degrees}\label{sectionchern}

The Grothendieck-Riemann-Roch theorem for singular varieties was proved by Baum, Fulton and MacPherson \cite{BFM75}, \cite{BFM79}
in the projective case and by Levy \cite{Levy1987} in the complex analytic case. 
One way to formulate it is that there exists a natural transformation of functors $\tau:K_0\to H_\Q:=H_{2*}( \ ;\Q)$ such that for any compact complex space $X$ the diagram 
$$
  \xymatrix{K^0X\otimes K_0X \ar[r]^{\otimes}\ar[d]^{\ch\otimes\tau} 
& K_0X \ar[d] ^{\tau}\\
H^{2*}(X;\Q)\otimes H_{2*}(X;\Q) \ar[r]^{ \ \ \ \ \ \ \frown} & H_{2*}(X;\Q)
 }
  $$ 
commutes and if $X$ is nonsingular then $\tau(\cO_X)=\Td(X)\frown[X]$, where
$K^0X$, $K_0X$ are the Grothendieck groups generated by holomorphic vector bundles and coherent sheaves respectively and $\Td(X)$ is the (cohomology) Todd class of the tangent bundle to $X$. Naturality means that for each proper morphism
$f:X\to Y$ of complex spaces the diagram
$$
  \xymatrix{ K_0X \ar[r]^{\tau}\ar[d]^{f_!} 
& H_{\Q}(X) \ar[d] ^{f_*}\\
K_0Y\ar[r]^{\tau} & H_{\Q}(Y)
 }
  $$ 
commutes, where $f_!$ is defined by $f_!([\cF])=\sum_i(-1)^i [R^if_*(\cF)]$ for any coherent sheaf $\cF$ on $X$. (In the non-compact case $\tau$ takes values in the Borel-Moore homology, \cite{Iversen}, \cite[19.1]{Fulton}.)  
It is  furthermore shown in \cite{Fulton} in the algebraic case and in \cite{Levy-bivariant} in the complex analytic case that $\tau$ may be extended as a Grothendieck transformation between bivariant theories in the sense of \cite{FultonMacPh}.
We refer to the original papers and to the books \cite{Fulton}, \cite{FL}, \cite{DV} for a thorough treatment of these facts.

For a coherent sheaf $\cF$ on a compact complex space $X$ we shall call $\tau(\cF):=\tau([\cF])$ the {\em homology Todd class} of $\cF$. 
The following Proposition gathers some of its properties, cf. \cite[Section 2]{TomaLimitareaI}. 
\begin{Prop} \label{prop:Todd}
\begin{enumerate}
\item If $\cF$ is locally free and $X$ is smooth and connected, $\tau(\cF)$ is the Poincar\'e dual of $ch(\cF)\cdot\Td(X)\in H^{2*}(X;\Q)$.
\item If $f:X\to Y$ is an embedding then $\tau(f_*\cF)=f_*(\tau(\cF))$.
\item $\tau(\cF)_r=0$ for $r>\dim \Supp\cF$.
 \item $\tau$ is additive on exact sequences.
\item If $X$ is irreducible then $\tau(\cF)_{\dim X}=\rank(\cF)[X]\in H_{2\dim X}(X;\Q) $.
\item  The component of  $\tau(\cF)$ in degree $\dim \Supp\cF$ is the homology class of an effective analytic cycle.
\end{enumerate}
\end{Prop}

\begin{Def}\label{def:multipolarization}
Let $X$ be an analytic space of dimension $n$ endowed with a system of differential forms $(\omega_1,...,\omega_{\dim(X)})$ such that each $\omega_{p}$ is $\d$-closed of degree $2p$ with   $(p.p)$-component $\omega_p^{p,p}$ which is strictly positive in Lelong's sense; cf. \cite[III.2.4, IV.10.6]{BarletMagnusson}. We call such an $(X,\omega_1,...,\omega_{\dim(X)})$ a {\em multi-polarized analytic space}. We will moreover always put $\omega_p=0$ if $p>\dim(X)$ and $\omega_0=1$. 
For any coherent sheaf $\cF$ with compact support on a multi-polarized analytic space $(X,\omega_1,...,\omega_{\dim(X)})$ and any non-negative $p$ we define the $p$-{\em degree}  of $\cF$ by 
$\deg_p(\cF):=\int_{\tau_p(\cF)}\omega_p$, where the integral is computed on a semianalytic representative of $\tau_p(\cF)$, cf. \cite{BlHe69}, \cite{DoPo77}, \cite[Thm. 7.22]{AG}, \cite[8.4]{Gor81},\cite{Her66}. 
\end{Def}

Occasionaly we will use positive $(p,p)$-forms instead of strictly positive forms. In such cases we will speak of {\em pseudo-degrees} and {\em pseudo-volumes} instead of degrees and volumes.

When $(X,\omega)$ is a K\"ahler space, cf. \cite[II.1.2]{Var89}, we will consider  its standard multi-polarization given by $(\omega,\omega^2,...,\omega^{\dim(X)})$. When
 $(X, \cO_X(1))$ is a projective variety endowed with an ample line bundle, by taking $\omega$ a strictly positive curvature form of $\cO_X(1)$, we recover the coefficient of the Hilbert polynomial of $\cF$ in degree $p$ as 
$$
\frac{\deg_p(\cF)}{p!}.
$$

One further property we shall need is the invariance of degrees in a flat family of coherent sheaves. For this we need some preparations. We start with a statement which is an analogue of \cite[Example 18.3.8]{Fulton} in our context. 
We shall work on the category $\cC$ of complex analytic spaces which have the topology of finite CW-complexes. For the bivariant theories we will consider on $\cC$ the  {\em confined morphisms} will be the proper morphisms  in $\cC$ and the {\em independent squares } will be the fiber squares in $\cC$; see \cite{FultonMacPh} for the terminology. We will use the two operational bivariant theories $K$ and $H$ on $\cC$ induced respectively by the homological theories $K_0$ and $H_\Q$ as in \cite[I.7]{FultonMacPh}. By \cite{Levy-bivariant} $\tau$ extends to a Grothendieck transformation between these bivariant theories, which we denote again by $\tau$. Restriction to fibers will be understood with respect to either of these bivariant theories. Note that a topological bivariant theory on $\cC$ may be also constructed from the standard cohomology theory and induces the usual homology theory, see \cite[Remark 3.1.9]{FultonMacPh}. In particular specialization in this theory as described in \cite[I.3.4.4]{FultonMacPh} induces specialization in the theory $H$; cf. \cite[I.8.2]{FultonMacPh}.

\begin{Prop} \label{plat}
 Let $X\to T$ be a proper morphism in $\cC$ with $T$ smooth and let $\cE$ be a coherent sheaf on $X$ flat over $T$. 
Then  $\tau(\cE_t)=\tau(\cE)_t$ in  $H_{\Q}(X)$ for all $t\in T$.
\end{Prop}
\begin{proof}
The sheaf $\cE$ defines a class $\alpha=[\cE]$ in $K(X\to T)$ in the folllowing way. For any fiber square
$$
 \xymatrix{X' \ar[r]^{g'} \ar[d]^{f'} & X \ar[d]^{f} \\
T' \ar[r]^{g} & T }
$$ 
we define $\alpha_{Y'}=[\cE]_{Y'}:K_0(Y')\to K_0(X')$ by  $\cF\mapsto \cF\otimes_Y\cE:=\cF_{X'}\otimes \cE_{X'}$. In order for $\alpha$ to be a class in $K(X\to T)$ the following compatibility conditions for the morphisms $\alpha_{Y'}$ need to be satisfied: if
$$
 \xymatrix{X''\ar[r]^{h'}\ar[d]^{f''} & X' \ar[r]^{g'} \ar[d]^{f'} & X \ar[d]^{f} \\
T''\ar[r]^h & T' \ar[r]^{g} & T }
$$ 
is a fiber diagram with $h$ proper, then 
$$
\xymatrix{K_0(T'') \ar[r]^{\alpha_{T''}} \ar[d]^{h_*} & K_0(X'') \ar[d]^{h'_*} \\
K_0(T') \ar[r]^{\alpha_{T'}} & K_0(X') }
$$
must commute. This holds because by the flatness of $\cE$ over $T$ one has canonical isomorphisms on $X'$
$$\cE_{X'}\otimes f'^*(R^ih_*\cF)\cong R^ih'_*(\cE_{X''}\otimes\cF_{X''})$$
for all coherent sheaves $\cF$ on $X''$ and all $i\in\N$. The proof of this fact goes exactly as that of \cite[Proposition III.9.3]{Hartshorne} using the flatness of $\cE_{X'}$ over $T'$ instead of  the flatness of $\cO_{X'}$.

Applying now the Grothendieck transformation $\tau:K\to H$ to $\alpha$, the class of $\cO_T$ and the map $t\to T$ gives the desired equality of classes. 
\end{proof}

\begin{Cor}\label{cor:constance}
Let $f:X\to T$ be a proper morphism of reduced irreducible complex spaces and let $\cE$ be a coherent sheaf on $X$ which is flat over $T$. Suppose that $X$ is endowed with a system $(\omega_j)_{j\ge1}$ of closed differential forms inducing multi-polarizations on each fibre $X_t$ of $f$. Then the degree functions $t \mapsto \deg_p(\cE_t)$ computed with respect to these polarizations are constant on $T$ for each positive $p$.
\end{Cor}

\begin{proof}
The statement is local on $T$. 
In fact for our purposes $T$ may be assumed to be the complex disc $\Delta$. 
Moreover using Thom's First Isotopy Lemma we may also assume that the induced map $X\setminus f^{-1}(0)\to \Delta\setminus\{ 0\}$ is a topologically locally trivial fibration, cf. \cite[1.3.6]{Dimca-Singularities}, \cite[p. 179]{Dimca-Sheaves}. By the Proposition $\tau(\cE_t)=\tau(\cE)_t$ in $H_\Q(X)$. We have to check that $\deg_p(\cE_t)=\deg_p(\cE_0)$ for any $t\in\Delta\setminus\{ 0\}$.  This follows now using the description of specialization of homology given in \cite[I.8.2]{FultonMacPh}, the fact that the forms $\omega_p$ are closed and Stokes' theorem in this context, see e.g. \cite[8.4]{Gor81}.
\end{proof}

Note that a system of closed differential forms as in the above Corollary is directly obtained as soon as the morphism $f$ is K\"ahler; see Definition \ref{def:kaehler-morphism} and comments thereafter. 

\subsection{Degree systems}
\label{subsection:degree-systems}

For technical reasons mainly related to our local volume computations the degrees of coherent sheaves used in this paper will be given by fixed positive differential forms on the base complex spaces. In this subsection we will nevertheless attempt at a more general definition which we believe to be naturally adapted to our considerations.
 
Let $X$ be a compact analytic space of dimension $n$ and let $d$, $d'$ be integers satisfying 
$n\ge d\ge d'\ge 0$. 
We denote by $[F]$ the class in $K_0(X)$ of a coherent sheaf $F$. If $F$ has dimension at most $p$, we write $\cycle_p(F)$ for the $p$-cycle associated to $F$.  

\begin{Def} Degrees.
\label{def:degrees}
A group morphism 
$\deg_p:K_0(X)\to\R$ will be called a 
{\em degree function in dimension $p$} on $X$ if it enjoys the following properties:
\begin{enumerate}
\item $\deg_p$ induces a positive map on non-zero $p$-cycles when putting $\deg_p(Z):=\deg_p([\cO_Z])$ for irreducible $p$-cycles $Z$
\item $\deg_p([F])=\deg_p(\cycle_p(F))$ for any coherent sheaf $F$ of dimension at most $p$ on $X$,
\item if a set of positive $p$-cycles is such that $\deg_p$ is bounded on it, then $\deg_p$ takes only finitely many values on this set, 
\item $\deg_p$ is 
continuous on flat families of sheaves.
\end{enumerate}
If $\deg_p$ is even locally constant 
on flat families of sheaves, we will say that it is a  {\em strong degree function}.
We will write for simplicity $\deg_p(F)=\deg_p([F])$ for any $F\in \Coh(X)$.

A collection $(\deg_d,...,\deg_{d'})$ of (strong) degree functions in dimensions $d$ to $d'$ on $X$ will be called a {\em (strong) $(d,d')$-degree system}. An $(n,0)$-degree system will be called a {\em  complete degree system}.
\end{Def}

Note that for any degree function $\deg_p$ in dimension $p$ on $X$ one has $\deg_p(F)>0$  if $F$ is $p$-dimensional, and $\deg_p(F)=0$ if $F$ is at most $(p-1)$-dimensional.

\begin{Def} Relative degrees.
\label{def:relative-degrees}
If $X$ is a complex space proper over a complex space $S$ a family $(\deg_{d,s},...,\deg_{d',s})_{s\in S}$ of degree systems parameterized by $S$ will be called a {\em relative degree system} on $X/S$ if each degree function $\deg_{\delta,s}$,  $d\ge\delta\ge d'$, 
 is 
continuous on flat families of sheaves over $S$, i.e. on flat families over any $S'$ after any base change $S'\to S$. Such a relative degree system will be said to be {\em strong} if its degree functions are locally constant 
on flat families of sheaves and it will be said to be {\em complete} if the dimensions of the degree functions range from the relative dimension of $X/S$ to $0$.
\end{Def}

One can also think of introducing degree systems depending on real parameters,  as in deformation theory for instance, where one defines a notion of a family of complex manifolds over a real manifold. In this paper we will only deal with product situations of this type. In such a case we 
define continuous families of relative degree systems as follows. If $X$ is a complex space proper over a complex space $S$ and if $T$ is a locally compact topological space a family $(\deg_{d,(s,t)},...,\deg_{d',(s,t)})_{(s,t)\in S\times T}$ of degree systems parameterized by $S\times T$ will be said to be a {\em continuous family of relative degree systems}  
if for each $t\in T$ the restricted family is a relative degree system on $X/S$ and for any  flat families of sheaves over $S$ the corresponding evaluations of the degree functions of the system  are continuous on $S\times T$.

Our basic example of a strong complete degree system is the one defined as in Definition \ref{def:multipolarization} on a multi-polarized analytic space $(X,\omega_1,...,\omega_{\dim(X)})$. In particular a compact K\"ahler space $(X,\omega)$ has a standard strong complete degree system.
 The degree functions of such systems satisfy condition (3) of Definition \ref{def:degrees} since they are constant on connected components of the corresponding cycle spaces and since no analytic space can have an infinite number of irreducible components accumulating at a point. The fact that they are constant on flat families of sheaves follows from Corollary \ref{cor:constance}. Strong complete relative degree systems appear in a situation as in the hypothesis of Corollary \ref{cor:constance}, in particular in the case of proper K\"ahler morphisms. Weakly K\"ahler morphisms give rise only to complete relative degree systems.

The condition on the positive differential forms defining degree functions of being $\d$-closed may be relaxed to asking only $\partial\bar\partial$-closure when one works on an ambient compact complex manifold. In such a situation using Property (1) of Proposition \ref{prop:Todd} one can replace in the definition of degrees homology Todd classes of coherent sheaves by Chern classes in Bott-Chern cohomology.  One can see that
 Condition (3) of Definition \ref{def:degrees} is  satisfied by the same argument as in the K\"ahler case. Indeed such a  function is pluriharmonic on the cycle space by \cite[Proposition 1]{BarletConvexitate} and attains its minimum on any closed subset of the cycle space by Bishop's theorem. It is therefore constant on any irreducible component of this space. In particular any smooth compact complex surface may be endowed with a complete degree system by fixing a Gauduchon metric on it.

Whenever degree functions or systems are induced by fixing strictly positive differential forms as in the above cases, we will say that they {\em come from differential forms}. The results of this paper will be established only for such degree systems, since essential use will be made of local volume computations as well as of (co)-homological projection formulas and duality as in Section \ref{section:quotient-sheaves}. It would be interesting to formulate abstract conditions for degree functions allowing to prove similar results.  

\section{Metric computations}\label{section:volume}

In this section we will consider positivity issues for some forms appearing naturally on total spaces of K\"ahler mappings, particularly on projectivized bundles.  
Positivity and strict positivity for $(p,p)$-forms on a manifold will be considered in the sense of Lelong, i.e. by asking non-negativity, respectively positivity, on each  complex tangent $p$-plane, cf. \cite[D\'efinition 10.6.1]{BarletMagnusson}. 

The main objective is to show that it is enough to bound a certain pseudo-volume of projectivized bundles in order to get a true volume bound. More precisely, we fix a compact irreducible analytic space $X$ of dimension $n$ and a coherent sheaf $G$ on $X$. (Actually the irreducibility assumption is made here only for simplicity.) In Proposition \ref{prop: forme pozitive pe spatii relative} we endow $\P(G)$ with a certain positive, not necessarily strictly positive, differential form $\Omega$ which depends on two fixed strictly positive forms $\omega_n$, $\omega_{n-1}$ on $X$ and on a metric $h$ on $G$. For each rank one quotient $F$ of $G$ we consider the pseudo-volume of the subspace $\P(F)$ of $\P(G)$ computed by integrating $\Omega$ on $\P(F)$. We want to show that a uniform bound of this pseudo-volume for all rank one quotients of $G$ leads to a uniform bound on a true volume of the subspaces $\P(F)\subset\P(G)$.  One problem appearing when trying to compute metrics on such objects is the fact that the sheaves involved and in particular $G$ may be singular. Moreover in general $G$ is not even a global quotient of a locally free sheaf. But it is so locally. Hence the idea to go to a local situation over some open subset $U$ of $X$ where a locally free sheaf $V$ exists together with a local epimorphism $V\to G_U$. Proposition \ref{prop: forme pozitive induse via surjectii} says that we can always lift a hermitian metric from $G_U$ to $V$. We will work with such a lifted metric on $V$ when we compute the  pseudo-volume of the restricted set $\P(F_U)\subset\P(G_U)\subset\P(V)$. If $V$ is of rank $r+1$ we associate to the quotient $V\to F_U$ the sheaf $F'=\Coker(\bigwedge^r\Ker(V\to F_U)\to\bigwedge^rV)$, show that the volume of $\P(F')$ is controlled by the pseudo-volume of $\P(F_U)$ and finally show that the volume of $\P(F')$  controls  a true volume of $\P(F_U)$.

\begin{Prop} \label{prop:pseudovolume}
 Let $X$ be an $n$-dimensional compact complex manifold and let $f:Y\to X$ be a smooth proper map of relative dimension $m>0$. 
 For $0\le p< m$ let $\eta_p$, ($\eta_{p+1}$), be a $(p,p)$-form, (respectively a $(p+1,p+1)$-form), on $Y$ which are strictly positive on the fibres of $Y\to X$ and let $\omega_{n-1}$, ($\omega_n$),    be  a strictly positive $(n-1,n-1)$-form,  (respectively a strictly positive $2n$-form), on $X$. Then there exist some positive constant $C$ such that the form
 $$\Omega:=Cf^*(\omega_n)\wedge\eta_{p}+f^*(\omega_{n-1})\wedge\eta_{p+1}$$
 is positive on $Y$ and strictly positive on any complex tangent $(n+p)$-plane whose projection on $X$ is at least $(n-1)$-dimensional.  
\end{Prop}
\begin{proof}
Let $y\in Y$ be a point and $P\subset T_yY$ be a complex $(n+p)$-plane tangent at $y$. If $f_*P$ is at most $(n-2)$-dimensional, then it is clear that both forms $
f^*(\omega_n)\wedge\eta_{p}$ and $f^*(\omega_{n-1})\wedge\eta_{p+1}$ vanish on $P$. 
If $f_*P$ is  $(n-1)$-dimensional, then  $
f^*(\omega_n)\wedge\eta_{p}$ vanishes on $P$ and $f^*(\omega_{n-1})\wedge\eta_{p+1}$ is strictly positive on $P$. So we are left with the case when $f_*P$ is $n$-dimensional.
In this case $
f^*(\omega_n)\wedge\eta_{p}$ is positive on $P$ but $f^*(\omega_{n-1})\wedge\eta_{p+1}$ not necessarily so because of possible negativity of $\eta_{p+1}$ on horizontal directions.

Since $X$ is compact the problem is local on the base $X$, so we may suppose that $X$ is a polydisc in $\C^n$.
 Since $f$ is proper the problem is also local on the fibres, so locally around $y$ we may suppose that $f$ is the projection from a polydisc $\Delta^n\times \Delta^m$ to $\Delta^n$. 
 Let $w_1,...,w_n$ and $z_1,...,z_m$ and $\omega:=\sum_{j=1}^ni\d w_j\wedge\d\bar{w}_j$, $\eta:=\sum_{j=1}^mi\d z_j\wedge\d\bar{z}_j$ be the coordinate functions and the standard hermitian forms on $\Delta^n$ and $\Delta^m$ respectively.  We set $v_j:=\frac{\partial}{\partial z_j}$, for $j=1,...,m$ and $h_k:= \frac{\partial}{\partial w_k}$, for $k=1,...,n$. 
 We may choose the coordinates $w_1,...,w_n$ so that at $f(y)$ we have $\omega_{n-1}=\omega^{n-1}$. Moreover we can compare the forms $\omega_n$ and $\omega^n$  and for our purposes we may suppose that  $\omega_n=c_1\omega^n$ for a positive constant $c_1$. 
We will further suppose for simplicity that our tangent $(n+p)$-plane $P$ is generated by the vectors $h_1+v'_1,... , h_n+v'_n,v_1,...,v_p$, where $v'_1,...,v'_n$ are vertical tangent vectors for $f$. For the next computation we shall identify $\bigwedge^{q,q}_\R(T^*_yY)$ to  the space of hermitian forms on $\bigwedge^q(T_yY)$, cf. \cite[Proposition 10.5.4]{BarletMagnusson}. We will also write $\hat{h}_j$ for $h_1\wedge...h_{j-1}\wedge h_{j+1}\wedge...\wedge h_n$. With these conventions we get 
$$\Omega((h_1+v'_1)\wedge...\wedge (h_n+v'_n)\wedge v_1\wedge...\wedge v_p,(h_1+v'_1)\wedge...\wedge (h_n+v'_n)\wedge v_1\wedge...\wedge v_p)=$$ $$C\omega_n(h_1\wedge...\wedge h_n,h_1\wedge...\wedge h_n)\eta_p(v_1\wedge...\wedge v_p,v_1\wedge...\wedge v_p)+$$ 
$$\sum_{j=1}^{n}\omega^{n-1}(\hat{h}_j,\hat{h}_j)\eta_{p+1}(v_1\wedge...\wedge v_p\wedge(h_j+v'_j), v_1\wedge...\wedge v_p\wedge(h_j+v'_j))=$$ $$Cc_1 \eta_p(v_1\wedge...\wedge v_p,v_1\wedge...\wedge v_p) + \sum_{j=1}^{n}\eta_{p+1}(v_1\wedge...\wedge v_p\wedge(h_j+v'_j), v_1\wedge...\wedge v_p\wedge(h_j+v'_j)).$$ 
We clearly have a positive lower bound on $\eta_p(v_1\wedge...\wedge v_p,v_1\wedge...\wedge v_p) $ and it remains to estimate 
$\eta_{p+1}(v_1\wedge...\wedge v_p\wedge(h_j+v'_j), v_1\wedge...\wedge v_p\wedge(h_j+v'_j)).$
But in the decomposition 
$\eta_{p+1}(v_1\wedge...\wedge v_p\wedge(h_j+v'_j), v_1\wedge...\wedge v_p\wedge(h_j+v'_j))=\eta_{p+1}(v_1\wedge...\wedge v_p\wedge h_j, v_1\wedge...\wedge v_p\wedge h_j)+2\Re e(\eta_{p+1}(v_1\wedge...\wedge v_p\wedge h_j, v_1\wedge...\wedge v_p\wedge v'_j))+ \eta_{p+1}(v_1\wedge...\wedge v_p\wedge v'_j, v_1\wedge...\wedge v_p\wedge v'_j)$, the first term is clearly bounded from below and also the sum of the following two terms, since the third term is positive and quadratic and the second one is linear in $v'_j$. So a uniform constant $C$ satisfying our requirements may be found.
\end{proof}

\begin{Prop} \label{prop: forme pozitive pe spatii relative}
 Let $X$ be an irreducible compact complex space of dimension $n$, let $G$ be a coherent sheaf of rank $m+1>1$ on $X$ and let $\eta$ be a positive $(1,1)$-form on the fibres of $\P(G)\to X$. Let further $\omega_{n-1}$, ($\omega_n$),    be  a strictly positive $(n-1,n-1)$-form,  (respectively a strictly positive $2n$-form), on $X$ and denote by $Z$ the irreducible component of $\P(G)$ covering $X$ and by $f:Z\to X$ the projection map.
 Then there exist some positive constant $C$ such that the form
 $$\Omega:=Cf^*(\omega_n)\wedge\eta^{p}+f^*(\omega_{n-1})\wedge\eta^{p+1}$$
 is positive on $Z$ and strictly positive on any complex tangent $(n+p)$-plane whose projection on $X$ is at least $(n-1)$-dimensional.  
\end{Prop}
\begin{proof}
The point is to show that   we can adapt the argument of   Proposition \ref{prop:pseudovolume} to the given situation and produce the desired constant $C$ even though the morphism $f:Z\to X$ is no longer supposed to be smooth.

There will be no restriction of generality if we suppose that the form $\eta$ on $\P(G)$ appears via a partition of unity $(\pi_i)_i$ subordinated to an open cover $(U_i)_i$  of $X$ using local embeddings $\P(G)_{U_i}\subset \P(H_i)$ and vertical positive forms $\eta_i$ on $\P(H_i)$, where $H_i$ are free $\cO_{U_i}$-modules of finite rank covering $G_{U_i}$ and $\eta_i$ are vertical Fubini-Study $(1,1)$-forms on $\P(H_i)$. 
The fact that the forms $\eta_i$ are defined on  $\P(H_i)$ allows us to use the inequalities obtained in the proof of  
Proposition \ref{prop:pseudovolume} over compact subsets $K_i$ contained in the interior sets of $\Supp(\phi_i)$. Indeed, the morphisms $\P(H_i)\to U_i$ are smooth and there exist positive constants $C_i$ such that on vertical $p$-planes $P$ tangent to $\P(G)_{K_i}$ comparison formulas
$\eta_{|P}\le C_i{\eta_i}_{|P}$ hold.
\end{proof}

\begin{Prop} \label{prop: forme pozitive induse via surjectii}
 Let $X$ be a complex space and let $U$ be a relatively compact open subset of $X$. Let further $H\to G$ be an epimorphism of coherent sheaves on $X$. Then any hermitian metric $\bar g$ on $G$ admits a hermitian lift to $H$ which is strictly positive over $U$. 
\end{Prop}
\begin{proof}
One can see that a hermitian lift $g$ of $\bar g$ exists using the definition of the metrics which are $C^\infty$ sections in the corresponding sheaves $G^{1,1}$ and $H^{1,1}$, see \cite[(4.2)]{Bin83}. The fact that this lift may be changed into one which is strictly positive over $U$ follows from \cite[Lemma 4.4]{Bin83}.  
\end{proof}

For the convenience of the reader we include next the computation of some easy integrals on $\C^n$ to be used later.

\begin{Lem}\label{lem: classical integrals} 
Let $(z_1,...z_n)$ be the standard coordinate functions on $\C^n$, put $z_0=1$, $z:=(1,z_1,...,z_n)$ and $|z|^2:=\sum_{j=0}^n|z_j|^2 =1+\sum_{j=1}^n|z_j|^2$. Let further $\omega$ be the restriction to $\C^n$ of the (standard) Fubini-Study metric on $\P^n$, where $\C^n$ is seen as a chart domain of $\P^n$, i.e. $\omega=\frac{i}{2\pi}\partial\bar\partial\log|z|^2$, and
$$I_{j\bar k}:=\int_{(\C^*)^n}\frac{z_j\bar z_k}{|z|^2}\omega^n, \ \text{for} \ j,k\in\{0,...,n\}.$$
Then $I_{j\bar k}$ equals $0$ for $j\neq k$ and $\frac{1}{n+1}$ for $j=k$.
\end{Lem}
\begin{proof}
We start with the case $j=k$. It is readily shown that $\int_{(\C^*)^n}\omega^n=1$, cf. \cite{DemaillyBook}. 
Thus $\sum_{j=0}^n I_{j\bar j}=1$. On $(\C^*)^n$ we use the coordinate change $f$, defined by $z_1=\frac{1}{w_1}$, $z_j=\frac{w_j}{w_1}$ for $j\notin\{0,1\}$, to get $|z|^2=\frac{|w|^2}{|w_1|^2}$, $f^*\omega=f^*(\frac{i}{2\pi}\partial\bar\partial\log(1+\sum_{j=1}^n|z_j|^2))=\frac{i}{2\pi}\partial\bar\partial\log(\frac{1}{|w_1|^2}{(1+\sum_{j=1}^n|w_j|^2}))=\frac{i}{2\pi}\partial\bar\partial\log|w|^2-\frac{i}{2\pi}\partial\bar\partial\log|w_1|^2=\frac{i}{2\pi}\partial\bar\partial\log|w|^2=:\omega_w$ and 
$$I_{0\bar 0}=\int_{(\C^*)^n}\frac{1}{|z|^2}\omega^n= \int_{(\C^*)^n}\frac{|w_1|^2}{|w|^2}\omega_w^n=I_{1\bar 1}, $$
hence $I_{j\bar j}=\frac{1}{n+1}$ for all $j\in\{0,...,n\}.$

We now look at the case $j\neq k$. The same substitution as in the case $j=k$ shows that $I_{0\bar 1}=I_{1\bar 0}$ and a similar computation also gives $I_{1\bar k}=I_{0\bar k}$ for all $k\notin\{0,1\}$. By symmetry all $I_{j\bar k}$ are thus equal when $j\neq k$.
Finally the orientation reversing coordinate change $z_1=\bar w_1$, $z_j=w_j$ for $j\notin\{0,1\}$ leads to $ I_{0\bar 1}=-I_{1\bar 0}=-I_{0\bar 1}$, whence our assertion.
\end{proof}

For any hermitian holomorphic line bundle $(L,h)$ on a complex manifold we will denote by $R^L$ its curvature form with respect to the Chern connection of $(L,h)$ and by $c_1(L,h):=\frac{i}{2\pi}R^L$ the associated Chern form. We have $R^L=-\partial\bar\partial \log|s|^2$, where $s$ is any (non-vanishing) holomorphic local section of $L$. If $L$ is the associated determinant line bundle of a holomorphic hermitian bundle $(E,h)$ and if $(h_{j\bar k})=(h(s_j,s_k))$ is the matrix of $h$ with respect to a local holomorphic frame $s=(s_1,..., s_r)$  of $E$, then $R^L= -\partial\bar\partial \log(\det(h_{j\bar k}))$.

The next Lemma serves in establishing a local comparison formula between (pseudo-) volumes of projective subbundles in the following situation. To a pure rank one quotient $F$ of a locally free sheaf $V$ of rank $r+1$ over a reduced complex space $X$ we associate a pure quotient $F'$ of rank $r$ of $\bigwedge^rV$ in a canonical way via the cokernel of the morphism $\bigwedge^r\Ker(V\to F)\to\bigwedge^rV$.  (Recall that a coherent sheaf $F$ is called {\em pure of dimension} $d$ if $\cF$ as well as all its non-trivial coherent subsheaves are of dimension $d$.) When $X$ and $V$ are endowed with hermitian metrics we get corresponding (pseudo-) volume forms on $\P(F)$ and $\P(F')$. We shall see that the corresponding (pseudo-) volume of $\P(F)$ computed over any open relatively compact subset $U$ of $X$ controls the volume of $\P(F')$ over the same subset $U$. Since we are interested in volumes of irreducible complex spaces it is enough to integrate volume forms over dense Zariski open subsets. Therefore we may and will assume that $X$ is smooth and that $F$ is locally free over $X$. 

Let $\omega_{n-1}$, $\omega_n$ be positive forms on $X$ of bidegrees $(n-1,n-1)$ and $(n,n)$ respectively. For simplicity of notation we will denote also by 
 $\omega_{n-1}$ and $\omega_n$ the pullbacks of these forms to $\P(F)$ and $\P(F')$. 
We will denote by $\eta_V$ the relative Fubini-Study $(1,1)$-form on $\P(V)$ with respect to a hermitian metric $h$ on $V$ and we will use the same symbol for its restriction to $\P(F)$, and likewise for the induced relative Fubini-Study $(1,1)$-form $\eta_{\bigwedge}$ on $\P(\bigwedge^rV)$ and for its restriction to $\P(F')$. 
 The (pseudo-)volume forms we consider on $\P(F)$ and on $\P(F')$ have each two components, namely $\omega_n$ and $\eta_V\wedge\omega_{n-1}$ on $\P(F)$, and $\eta^{r-1}_{\bigwedge}\wedge\omega_n$ and $\eta_{\bigwedge}^r\wedge\omega_{n-1}$ on $\P(F')$.  

The comparison of $\omega_n$  on $\P(F)$ to $\eta^{r-1}_{\bigwedge}\wedge\omega_n$ on $\P(F')$ is immediately done since their fibre integrals with respect to the projections $\P(F)\to X$ and $\P(F')\to X$ are both equal to $\omega_n$ by the projection formula for fibre integrals, \cite{Stoll-fiber-integration}. In the sequel we denote fibre integrals by $\int_{Y\to X}$. By the projection formula again we get 
$$\int_{\P(F)\to X}(\eta_V\wedge\omega_{n-1})=(\int_{\P(F)\to X}\eta_V)\wedge\omega_{n-1}$$ 
and 
$$\int_{\P(F')\to X}(\eta_{\bigwedge}^r\wedge\omega_{n-1})=(\int_{\P(F')\to X}\eta_{\bigwedge}^r)\wedge\omega_{n-1}$$
so we are left with the task of comparing $\int_{\P(F)\to X}\eta_V$ to $\int_{\P(F')\to X}\eta_{\bigwedge}^r$. The result is 

\begin{Lem}\label{lem:fibre-integrals} 
In the above set-up we have
$$\int_{\P(F)\to X}\eta_V=\int_{\P(F')\to X}\eta_{\bigwedge}^r-(r-1) \frac{i}{2\pi} R^{\det V}.$$
\end{Lem}
\begin{proof}
Any $(1,1)$-form on $X$ is determined by its restriction to the smooth curves on $X$. Moreover fibre integration is compatible with restriction, cf. \cite{Stoll-fiber-integration}, so we may assume that $X$ is smooth one-dimensional. The property we want to prove is local on the base so we will consider a point $x$ of $X$ around which $V$ is trivial and endowed with a hermitian metric $h^V$. Let $s=(s_0,s_1,...,s_r)$ be a normal holomorphic local frame of $(V,h^V)$ at $x$, i.e. such that with respect to it we have $h^V_{i\bar j}(x)=\delta_{ij}$ and $(\d h^V_{i\bar j})(x)=0$ for all $i,j\in\{0,...,r\}$, cf. \cite[Proposition 1.4.20]{KobayashiVectorBundles}.

We start by computing  the ${(1,1)}$-form $\eta_V$ on $\P(F)$. We may consider $\P(F)$ as a subspace of $\P(V)$ but also as a subspace of $\P_{sub}(V^*)$ using the canonical identification $\P(V)\cong \P_{sub}(V^*)$. For our metric computation we will prefer the latter point of view. Here $\P_{sub}(V^*)$ is the projective bundle parameterizing $1$-dimensional subspaces in the fibers of $V^*$. The form $\eta_V$ on $\P_{sub}(V^*)$ is the curvature of the tautological quotient $\cO(1)$ on $\P_{sub}(V^*)$ and is computed as follows, see also \cite[proof of Theorem 3.6.17]{KobayashiVectorBundles},  \cite[Section 5.1.1]{MaMarinescu}, \cite{GriffithsHarris} or \cite{DemaillyBook}. The metric $h^V$ on $V$ induces a metric $h^{V^*}$ on $V^*$, a further metric $h^{\cO(-1)}$ on the tautological subbundle $\cO(-1)$ on $\P_{sub}(V^*)$ and a metric $h^{\cO(1)}$ on the dual line bundle $\cO(1)$. A holomorphic section $v$ in $V$ defines a fibrewise linear functional on $V^*$ by $f\mapsto (f,v)$, hence a section $\sigma_v$ in $\cO(1)$ on $\P_{sub}(V^*)$. If $f(x)\neq0$ one gets around the point $[f(x)]\in\P_{sub}(V^*)_x$
$$|\sigma_v([f])|^2_{h^{\cO(1)}}=\frac{|(f,v)|^2}{|f|^2_{h^{V^*}}}$$
and if $v$ doesn't vanish at $f(x)$
$$\eta_V=\frac{i}{2\pi}R^{\cO(1)}=-\frac{i}{2\pi}\partial\bar\partial\log|\sigma_v([f])|^2_{h^{\cO(1)}},$$
where $R^{\cO(1)}$ denotes the curvature of the Chern connection of ${\cO(1)}$ on $\P_{sub}(V^*)$.
The chosen holomorphic frame of $V$ trivializes $V$ giving $V\cong X\times V_x\cong X\times \C^{r+1}$. Let $e_j=s_j(x)$ and denote by $(e_j^*)_j$ the dual base of $(e_j)_j$. Denote by  $z$ a local coordinate function on $X$ at $x$. With respect to the above trivialization of $V$ the form $\eta_V$ decomposes as a sum of a vertical component $\eta_{V,vert}$, which is the corresponding Fubini-Study form in that fiber, a horizontal component $\eta_{V,hor}$, and a mixed component, which is immediately seen to vanish at all points lying over $x$ by the normality of the chosen frame $s$. At the point $(x,(1:0:...:0))$ a direct computation shows that
$$\eta_{V,hor}=\frac{i}{2\pi}\partial\bar\partial \log h^{V^*}_{0,\bar 0}=\frac{i}{2\pi(h^{V^*}_{0,\bar 0})^2 }\partial\bar\partial  h^{V^*}_{0,\bar 0}= \frac{i}{2\pi}\frac{\partial^2 h^{V^*}_{0,\bar 0}}{\partial z\partial\bar z}\d z\d\bar z.$$
We view $\P(F)$ as the image of a section $\sigma:X\to\P_{sub}(V^*)\cong X\times\P_{sub}^r$ passing through the point  $(x,[1:0:...:0])\in X\times\P_{sub}^r$ and given by $\sigma(z)=(z,[f(z)])$, 
\begin{equation}\label{eq:parametrizare}
f(z)=e_0^*+f_1(z)e_1^*+...f_r(z)e_r^*, 
\end{equation}
where the functions $f_j$ are holomorphic and vanish at $0$. Then 
$\int_{\P(F)\to X}\eta_V=\sigma^*\eta_V$, which at $x$ gives
$$ \sigma^*\eta_V= \sigma^*\eta_{V,hor}+\sigma^*\eta_{V,vert}= 
\frac{i}{2\pi}\partial\bar\partial \log h^{V^*}_{0,\bar 0}
+\frac{i}{2\pi}\partial\bar\partial \log(1+\sum_{j=1}^r|f_j)|^2)=
$$ 
$$\frac{i}{2\pi}(\frac{\partial^2 h^{V^*}_{0,\bar 0}}{\partial z\partial\bar z}+\sum_{j=1}^r|f_j'(0)|^2)\d z\d\bar z=
\frac{i}{2\pi}(-\frac{\partial^2 h^{V}_{0,\bar 0}}{\partial z\partial\bar z}+\sum_{j=1}^r|f_j'(0)|^2)\d z\d\bar z.$$

We will next compute $\eta_{\bigwedge}$ by working on $\P(V^*)$, which is isomorphic to $\P(\bigwedge ^rV)$ via the canonical isometry $\bigwedge^rV\cong V^*\otimes\det V$. We will first compute $\eta_{V^*}$ and then $\eta_{\bigwedge}=\eta_{V^*}+\frac{i}{2\pi}R^{\det V}=\eta_{V^*,vert}+\eta_{V^*,hor}+\frac{i}{2\pi}R^{\det V}$ at points over $x$. To do this we will use the isomorphism $\P(V^*)\cong\P_{sub}(V)$. Let $E:=\Ker(V\to F)$. We are interested in the restriction of 
$$\eta_{\bigwedge}^r=
\eta_{V^*,vert}^{r}+r\eta_{V^*,vert}^{r-1}\wedge\eta_{V^*,hor}+r\eta_{V^*,vert}^{r-1}\wedge\frac{i}{2\pi}R^{\det V}$$ 
to the subbundle $\P_{sub}(E)$ of $\P_{sub}(V)$ and in its fiber integral with respect to $\P_{sub}(E)\to X$. 

We compute the integrals of the three terms appearing on the right handside of the above expression of $\eta_{\bigwedge}^r$. The third integral gives 
$$\frac{ri}{2\pi}R^{\det V}=-
\frac{ri}{2\pi}\sum_{j=0}^r\partial_z\bar\partial_z h_{j,\bar j}^V.$$

Over $x$ with respect to our fixed frame $s$ the subspace $E_x$ of $V_x$ is given by the condition $e^*_0=0$. Consider the standard chart on $\P(E_x)$ corresponding to $e^*_1=1$ and set $w=(1,w_2,...,w_r)$, $w_1=1$, as "coordinate functions" similarly to Lemma \ref{lem: classical integrals}. Then at such a point we get
$$\eta_{V^*,hor}=
-\frac{i}{2\pi}\partial_z\bar\partial_z \log(\frac{1}{|w|^2_{h^V}})
=
\frac{i}{2\pi}\partial_z\bar\partial_z \log(\sum_{j=1}^r \sum_{k=1}^rh_{j\bar k}^Vw_j\bar w_k)=
$$
$$
\frac{i}{2\pi}\frac{\sum_{j=1}^r \sum_{k=1}^rw_j\bar w_k \partial_z\bar\partial_z  h_{j\bar k}^V
}{|w|^2}.
$$
Hence taking Lemma \ref{lem: classical integrals} into account the second integral over the fibres of $\P_{sub}(E)$ gives at $x$
$$\frac{i}{2\pi}\sum_{j=1}^r\partial_z\bar\partial_z  h_{j\bar j}^V=
-\frac{i}{2\pi}R^{\det V}-\frac{i}{2\pi}\partial_z\bar\partial_z  h_{0\bar 0}
.$$

In order to compute the first fibre integral we first parameterize $\P_{sub}(E)$ using the map
\begin{equation}\label{eq:parametrizare-duala}
\tau:X\times \P^{r-1}\to X\times \P^r, \ (x,[v_1:...:v_r])\mapsto (z, [-\sum_{j=1}^rv_jf_j(z):v_1:...:v_r]),
\end{equation}
where the functions $f_j$ are those appearing in formula \eqref{eq:parametrizare}.
The pull-back of $\eta_{V^*,vert}$ to $X\times \P^{r-1}$ over $x$ through this map equals
$$\tau^*\eta_{V^*,vert}=
\frac{i}{2\pi}\partial\bar\partial  \log (|\sum_{j=1}^rf_j(z)v_j|^2+|v_1|^2+ ...+|v_r|^2)$$
and we may as before integrate over the standard chart domain $\C^{r-1}$ where $v_1=1$. We set $w_1=1$ and let $w_2,...., w_r$ be the standard coordinate functions on $\C^{r-1}$ similarly to Lemma \ref{lem: classical integrals}. We also write $\d w_1=0$. Then we get 
$$
\tau^*\eta_{V^*,vert}=
\frac{i}{2\pi}\partial\bar\partial  \log (|\sum_{j=1}^rf_j(z)w_j|^2+|w|^2)= 
$$
$$
\frac{i}{2\pi}\partial \frac{(\sum_{j=1}^rf_j(z)w_j)((\sum_{j=1}^r\bar f'_j(z)\bar w_j)\d \bar z+\sum_{j=1}^r\bar f_j(z)\d \bar w_j) +\bar\partial |w|^2 }{|\sum_{j=1}^rf_j(z)w_j|^2+|w|^2}=
$$
$$
\frac{i}{2\pi} (\frac{|\sum_{j=1}^rf'_j(z)w_j|^2\d z\d \bar z+\partial\bar\partial |w|^2 }{|w|^2}-
\frac{\partial |w|^2\wedge \bar\partial |w|^2 }{|w|^4})=
$$
$$
\frac{i}{2\pi} \frac{|\sum_{j=1}^rf'_j(z)w_j|^2\d z\wedge  \bar z }{|w|^2}+\omega_{FS}
$$
and
$$
\tau^*\eta_{V^*,vert}^r=\frac{ri}{2\pi} \frac{|\sum_{j=1}^rf'_j(z)w_j|^2\d z\d \bar z }{|w|^2}\wedge\omega_{FS}^{r-1}
$$ 
which integrated in the fiber over $x$ gives
$$ 
\frac{i}{2\pi} \sum_{j=1}^r|f'_j(0)|^2\d z\d \bar z
$$
by use of Lemma \ref{lem: classical integrals}.

Summing the three integrals up we get the desired statement.
\end{proof}

We next change slightly the set-up of Lemma \ref{lem:fibre-integrals} and specialize to the case when the metric $h$ on $V$ is trivial. To compute a true  volume of $\P(F)$ it is enough to know  the  fibre integrals of higher powers of $\eta_V$ and wedge them with appropriate powers  of a K\"ahler form $\omega$ of a hermitian metric on $X$. In order to do this we start by computing the form
$$\alpha:=\int_{\P(F)\to X}\eta_V.$$

\begin{Lem}\label{lem:higher-fibre-integrals} 
In the above set-up if the metric $h$ on $V$ is trivial, we have 
$$\alpha:=\int_{\P(F)\to X}\eta_V=\int_{\P(F')\to X}\eta_{\bigwedge}^{r}$$
and
for all positive integers $p$:
$$\int_{\P(F)\to X}\eta_V^p=\alpha^{p} \ \text{and} \  \int_{\P(F')\to X}\eta_{\bigwedge}^{r+p}=0.$$
\end{Lem}
\begin{proof}
The first assertion is a direct consequence of Lemma \ref{lem:fibre-integrals}. 
For the second one we use the same approach and notations as in the proof of Lemma \ref{lem:fibre-integrals} with the difference that the metric $h$ is supposed to be trivial and that the dimension of $X$ is some arbitrary positive integer $n$. Then using again a section $\sigma:X\to\P_{sub}(V^*)\cong X\times\P_{sub}^r$ passing through the point  $(x,[1:0:...:0])\in X\times\P_{sub}^r$ and given by $\sigma(z)=(z,[f(z)])$, with $f$ as in equation \eqref{eq:parametrizare}
 to parameterize $\P(F)$ locally at $x$ and putting $f_0(z)=1$ we get 
 $$ \sigma^*(\eta_V)^p= (\sigma^*(\eta_V))^p=
(\frac{i}{2\pi}\partial\bar\partial \log(\sum_{j=0}^r|f_j|^2))^p.
$$ 
Thus 
$$\alpha^p=(\frac{i}{2\pi}\partial\bar\partial \log(\sum_{j=0}^r|f_j|^2))^p=\int_{\P(F)\to X}\eta_V^p.$$
Finally, using the parameterization \eqref{eq:parametrizare-duala} we obtain over $x$
$$
\tau^*\eta_{\bigwedge}=
\frac{i}{2\pi} \frac{\partial_z(\sum_{j=1}^rw_jf_j) \wedge \overline{ \partial_z(\sum_{j=1}^rw_jf_j) }}{|w|^2}+\omega_{FS},
$$
whence for non-negative $p$ the formula
$$\int_{X\times\P^{r-1}\to X}
(\tau^*\eta_{\bigwedge})^{r+p}=
$$
$$
\binom{r+p}{p+1}
\int_{X\times\P^{r-1}\to X}
(\frac{i}{2\pi} \frac{\partial_z(\sum_{j=1}^rw_jf_j) \wedge \overline{ \partial_z(\sum_{j=1}^rw_jf_j) }}{|w|^2})^{p+1}\wedge\omega_{FS}^{r-1},
$$
whose integrand vanishes for $p>0$.
\end{proof}

 To complete the passage of volume computation to $\P(F)$ back from $\P(F')$ it will be therefore enough to bound integrals of the type $\int_X\alpha^p\wedge \omega^{n-p}$ in terms of $\int_X\alpha\wedge \omega^{n-1}$ and $\int_X \omega^{n}$. This is done using  inequalities of Hodge Index type, cf. \cite[Remark 5.3]{DemaillyJDG93}. For this approach it is important to note that the form $\alpha$ is positive on $X$.

\section{Degrees of quotient sheaves and volumes}\label{section:quotient-sheaves}

This section can be seen as part of the proof of the Key Lemma (Lemma \ref{tehnica}). In it we relate degrees of coherent quotient sheaves of a given coherent sheaf $G$ on a reduced irreducible compact analytic space $X$ to (pseudo-) volumes of projectivized bundles. Let $n$ be the dimension of $X$. We will fix degree functions $\deg_n$, $\deg_{n-1}$, coming from differential forms $\omega_n$, $\omega_{n-1}$, as in Section \ref{subsection:degree-systems}. For simplicity we assume that $\omega_n$, $\omega_{n-1}$ are the $(n,n)$-, respectively the $(n-1,n-1)$-, components of $\d$-closed forms, but the arguments can be easily adapted to the case when $X$ is embedded in some complex manifold $X'$ and $\omega_n$, $\omega_{n-1}$ are restrictions of $\partial\bar\partial$-closed forms on $X'$. 
Unlike in \cite{TomaLimitareaI} we choose in this paper to reduce ourselves to quotients of rank one. This is done by taking top exterior powers. We start with a pure $n$-dimensional quotient $F$ of rank $r$ of $G$ and perform on it the following operations which correspond to four steps of the proof of Lemma \ref{tehnica}:
\begin{enumerate}
\item 
Go from $F$ to a rank one coherent sheaf $F_2:=(\bigwedge^r F)_{pure}$ on $X$. Here we denote by $E_{pure}$ the pure $d$-dimensional part of a $d$-dimensional coherent sheaf $E$. Show that $\deg_{n-1}F_2$, $\deg_nF_2$ are uniformly controlled by $\deg_{n-1}F$, $\deg_nF$, in a way depending on $G$ and $X$ only.
\item 
Go from $F_2$ to the irreducible component $P_2$ of $\P(F_2)$ covering the base $X$. Show that the pseudo-volume of $P_2$ defined as in Proposition \ref{prop:pseudovolume} is controlled by 
$\deg_{n-1}F_2$, $\deg_nF_2$.
\item 
Go from $P_2$ to an injective morphism $F_2\to F_1$ into a coherent sheaf $F_1$ of rank one on $X$.  
\item
Recover $F$ from the morphism $\bigwedge^r G\to F_1$.
\end{enumerate}  

{\bf Step 1.} By property (5) of Proposition \ref{prop:Todd} the $n$-dimensional degree of a coherent sheaf $F$ on $X$ equals 
$\rank(F)\vol_{\omega_n}(X)$. If $F$ is a quotient of $G$ this quantity is completely controlled by $X$ and by $G$. Thus $\deg_nF_2=\vol_{\omega_n}(X)\le \deg_nF$. In order to obtain controll over 
$\deg_{n-1}F_2$ too, we first normalize X and then desingularize. Let $f:Y\to X$ be the normalization map, let $g:Z\to Y$ be a desingularization of $Y$ and put $h:=f\circ g$.  
Let $F_2$ be the image of the morphism $  \bigwedge^r G\to (\bigwedge^r F)_{pure}$. It is also the pure part of the image of the morphism $  \bigwedge^r G\to \bigwedge^r F$. 

The following Lemma implies that the pseudo-degree $\deg_{f^*\omega_{n-1}}f^*F$ is bounded in terms of $\deg_{n-1}F$.

\begin{Lem} \label{lem:normalization control} Let $f:Y\to X $ be  a finite morphism, $F$ be a pure coherent sheaf on $X$, which is a quotient of a (pure) coherent sheaf $G$ and let $G'':=\Coker(G\to f_*f^*G)$. Fix a closed positive $(n-1,n-1)$-form $\Omega$ on $X$, where $n=\dim X$. Then
$$\deg_{n-1,f^*\Omega}(f^*F)\le \deg_{n-1,\Omega}(F)
+ \deg_{n-1,\Omega}(G'').$$
\end{Lem} 
\begin{proof}
Set $F'':=\Coker(F\to f_*f^*F)$. 
By the projection formula and Grothendieck-Riemann-Roch we have
$$
\deg_{n-1,f^*\Omega}(f^*F)=$$ $$\tau_{n-1}(f^*F)f^*\Omega=
f_*(\tau_{n-1}(f^*F))\Omega=
\tau_{n-1}(f_*f^*F)\Omega=\deg_{n-1,\Omega}(f_*f^*F)$$
and the desired inequality follows from the commutative diagram with exact rows and columns below

\begin{equation}
\label{diagr:normalization}
\begin{gathered}
\xymatrix{ 
0 \ar[r]&  G \ar[r]\ar[d] & f_*f^*G\ar[d]\ar[r] & G'' \ar[d]\ar[r] & 0 \\
0 \ar[r] &  F\ar[r]\ar[d] & f_*f^*F \ar[d]\ar[r] & F''\ar[r] \ar[d] & 0 \\
 & 0  & 0  & 0 & 
}
\end{gathered}
\end{equation}
\end{proof}

By applying the projection formula, Grothendieck-Riemann-Roch and the fact that for a coherent sheaf $E$ on $Y$ the sheaves $E$ and $g_*g^*E$ may differ only in codimension larger than one we get
$$\deg_{g^*f^*\omega_{n-1}}h^*F= \deg_{f^*\omega_{n-1}}g_*g^*f^*F=\deg_{f^*\omega_{n-1}}f^*F.
$$
If $E$ is a coherent sheaf of rank $r$ on the non-singular space $Z$ and $\Omega$ is a closed $(n-1,n-1)$-form on $Z$, we have 
$$\deg_\Omega E\ge\deg_\Omega E_{pure}=\deg_\Omega \bigwedge^r(E_{pure})=\deg_\Omega (\bigwedge^rE)_{pure},$$
the last equality holding due to the fact that the sheaves $\bigwedge^r(E_{pure})$ and $(\bigwedge^rE)_{pure}$ may differ only in codimension larger than one. In particular we get
$$\deg_{f^*\omega_{n-1}}f^*F=\deg_{h^*\omega_{n-1}}h^*F\ge \deg_{h^*\omega_{n-1}}(\bigwedge^rh^*F)_{pure}.
$$
Consider now the following commutative diagrams

\begin{equation}
\label{diagr:exterior-powers}
\begin{gathered}
\xymatrix{ 
h^*\bigwedge^rG \ar[r]^\cong& \bigwedge^rh^* G \ar@{->>}[r] &\bigwedge^r h^*F\ar[r]^\cong & h^*\bigwedge^rF \ar@{->>}[d]\ar@{->>}[r] & (\bigwedge^rh^*F)_{pure} \\
 &  & & h^* F_2\ar@{->>}[ur]  & 
}
\end{gathered}
\end{equation}

\begin{equation}
\label{diagr:push-forward}
\begin{gathered}
\xymatrix{ 
  \bigwedge^rG \ar@{->>}[r]\ar[d] & \bigwedge^rF\ar[d]\ar@{->>}[r] & F_2 \ar[d]&  \\
  h_*h^*\bigwedge^rG\ar[r] & h_*h^*\bigwedge^rF \ar[r] & h_*h^*F_2\ar[r]  & h_*((\bigwedge^rh^*F)_{pure})
  }
\end{gathered}
\end{equation}
and note that $(\bigwedge^rh^*F)_{pure}$ is isomorphic to $(h^*F_2)_{pure}$. After applying $g_*$ we get an exact sequence
$$0\to g_*(\Tors(h^*F_2))\to g_*h^*F_2\to g_*((\bigwedge^rh^*F)_{pure})\to R^1g_*(\Tors(h^*F_2))$$
so the morphism $g_*h^*F_2\to g_*((\bigwedge^rh^*F)_{pure})$ is an isomorphism in codimension one and it will remain so after application of $f_*$. Moreover the natural morphism $F_2\to h_*h^*F_2$ is injective since $F_2$ is pure and thus
$$\deg_{\omega_{n-1}} F_2\le \deg_{\omega_{n-1}} h_*h^*F_2=
\deg_{\omega_{n-1}} h_*((\bigwedge^rh^*F)_{pure})=$$ $$
 \deg_{h^*\omega_{n-1}}(\bigwedge^rh^*F)_{pure}\le
\deg_{f^*\omega_{n-1}}f^*F$$
and the last term is controlled in terms of $\deg_{\omega_{n-1}} F$, $G$ and $X$.

{\bf Step 2.} Let $P_2$ be the irreducible component of $\P(F_2)$ which covers $X$. The projection $p:P_2\to X$ is bimeromorphic so $\int_{P_2}p^*\omega_n=\int_X\omega_n=\deg_{\omega_n}F_2$, which settles the $\omega_n$-component of the pseudo-metric on $P_2$. For the $\omega_{n-1}$-component it will be enough to show that $\deg_{p^*\omega_{n-1}}\cO_{P_2}(1)$ is controlled by $\deg_{\omega_{n-1}}F_2$. We consider the cartesian square
\begin{equation}
\label{diagr:cartesian}
\begin{gathered}
\xymatrix{ 
  P_{2,Y}\ar[d]^{p'}\ar[r]^{f'} & P_2 \ar[d]^p  \\
  Y \ar[r]^f & X,
  }
\end{gathered}
\end{equation}
where $f$ is as before the normalization map. Now 
$$\deg_{p^*\omega_{n-1}}\cO_{P_2}(1)\le \deg_{p^*\omega_{n-1}}f'_*f'^*\cO_{P_2}(1)=\deg_{f'^*p^*\omega_{n-1}}f'^*\cO_{P_2}(1)
$$
and 
$f'^*\cO_{P_2}(1)$ is a quotient of $f'^*p^*F_2=p'^*f^*F_2$ 
so $$\deg_{f'^*p^*\omega_{n-1}}f'^*\cO_{P_2}(1)\le \deg_{f'^*p^*\omega_{n-1}}f'^*p^*F_2=\deg_{f^*\omega_{n-1}}p'_*f'^*p^*F_2=$$ $$\deg_{f^*\omega_{n-1}}p'_*p'^*f^*F_2=\deg_{f^*\omega_{n-1}}f^*F_2
$$
and the last term is controlled by $\deg_{\omega_{n-1}}F_2$ by Lemma \ref{lem:normalization control}.

{\bf Step 3.}
Under the previous notations set $F_1:=p_*\cO_{P_2}(1)$. The desired injective morphism is given by the composition of the natural morphisms $F_2\to p_*p^*F_2\to p_*\cO_{P_2}(1)=F_1$. 

{\bf Step 4.} The recovery of $F$ is a consequence fo the following Lemmata.
\begin{Lem}\label{lem:restriction} (Pure quotients are determined by their restriction to an open dense subset.) Let $G$ be a coherent sheaf on a pure $n$-dimensional space $X$ and let $U\subset X$ be a dense Zariski open subspace. A pure $d$-dimensional quotient $G\to F$ of $G$ is then completely determined by its restriction to $U$.
\end{Lem} 
\begin{proof}
Suppose that two quotients $G\to F_1$, $G\to F_2$ with $F_1$, $F_2$ pure $n$-dimensional have the same restriction to $U$. Set $E_j:=\Ker(G\to F_j)$, $1\le j\le2$, and $E:=E_1\cap E_2$. From the diagrams
\begin{equation*}
\label{diagr: pure}
\begin{gathered}
\xymatrix{ 
&  &  & C_j \ar@{^(->}[d] &  \\
0 \ar[r] & E \ar[r]\ar@{^(->}[d] & G \ar[d]\ar[r] & G/E \ar[r] \ar@{->>}[d] & 0 \\
0 \ar[r] & E_j \ar[r]\ar@{->>}[d] & G \ar[r] & F_j \ar[r] & 0 \\
 & C_j & & & \\
}
\end{gathered}
\end{equation*}
for $j\in\{1,2\}$
we see that the quotients $G\to F_j$  both appear as the saturation of $G\to G/E$.
\end{proof}

\begin{Lem}\label{lem:determinant_trick} (determinant trick)
Let $A_:=\cO_{X,x}$ be the local ring of an analytic space of pure dimension $d$ and let $F$ be a pure $A$-module of dimension $n$ and of rank $r$ over each of its associated points. If $F$ appears as a quotient $G\to F$ of an $A$-module of finite type $G$, then $F$ can be completely recovered from the composition of morphisms
$$G\to \Hom(\bigwedge^{r-1}G,\bigwedge^rG)\to \Hom(\bigwedge^{r-1}G,(\bigwedge^rF)_{pure}),$$
where the first morphism $\phi:G\to \Hom(\bigwedge^{r-1}G,\bigwedge^rG)$ is given by taking exterior product, $g\mapsto(g_1\wedge...\wedge g_{r-1}\mapsto g_1\wedge...\wedge g_{r-1}\wedge g)$, and the second one $\psi: \Hom(\bigwedge^{r-1}G,\bigwedge^rG)\to\Hom(\bigwedge^{r-1}G,(\bigwedge^rF)_{pure})$ is composition with the natural morphism $\bigwedge^{r}G\to(\bigwedge^rF)_{pure}$. More precisely, one has $F\cong\Im(\psi\circ\phi)$.
\end{Lem}

\begin{proof}
We will show that $\ker(\psi\circ\phi)=\ker(G\to F)$.

Set $E:=\Ker(G\to F)$. It is clear that $E\subset \Ker(\psi\circ\phi)$. We then get a commutative diagram with exact rows and columns.
\begin{equation*}
\label{diagr: on X}
\begin{gathered}
\xymatrix{ 
&  &  & C \ar@{^(->}[d] &  \\
0 \ar[r] & E \ar[r]\ar@{^(->}[d] & G \ar[d]\ar[r] & F \ar[r] \ar@{->>}[d] & 0 \\
0 \ar[r] & \Ker(\psi\circ\phi) \ar[r]\ar@{->>}[d] & G \ar[r] & \Im(\psi\circ\phi) \ar[r] & 0 \\
 & C & & & \\
}
\end{gathered}
\end{equation*}
Since the property we want to show is valid at the generic point of each associated component of $F$ (where $F$ is free of rank $r$) it follows that $C:=\Coker(E\to \Ker(\psi\circ\phi))$ is at most $n-1$ dimensional and we conclude by the purity of $F$.
\end{proof}


\section{Bounded sets of coherent sheaves}\label{sec:boundedness}

Starting from this section all complex spaces will be supposed to be Hausdorff and second countable, i.e. allowing a countable base for their topology. In particular they will be $\sigma$-compact, cf. \cite[Theorem 12.12]{Bredon}. If $X$ is an analytic space over $S$ with projection morphism $p:X\to S$ and $F$ is a coherent sheaf on $X$ and if $T \to S$ is a morphism, we will write as usual $X_T:= X\times_S T$ and $F_T$ for the base change. The projections $X_T\to T$ will be denoted by $p_T$. For a {\em germ of an analytic set around a compact set} $K$ in the sense of \cite[VII.2(b)]{BS} we will use the notation $\tilde K$.
When speaking of morphisms or sheaves defined on a germ $\tilde K$ of an analytic space around a compact set $K$  we mean of course that such objects are defined on some analytic space containing $K$ and representing the germ  $\tilde K$.

If $X$ is a scheme of finite type over a noetherian scheme $S$, then according to \cite{Grothendieck-theHilbertScheme}  one can roughly say that a set of isomorphism classes of coherent sheaves on the fibres of $X/S$ is {\em bounded} if its elements are among the fibres of an algebraic family of coherent sheaves parametrized by a scheme $S'$ of finite type over $S$. It is not directly clear how to formulate a definition of boundedness in a complex geometrical set-up, which would be also effective in proving properness statements. 
In \cite{TomaLimitareaI} we gave such a formulation in the relative case  but used it in an absolute setting essentially. An equivalent way to formulate that definition in the absolute case is the following. A minor change made here as compared to \cite{TomaLimitareaI} is that we no longer ask that the compact set $K$ be a semi-analytic Stein compact set. One can reduce oneself to a situation where this property is satisfied and this is important in some of the arguments, but it is 
not necessary to include it in the definition.

\begin{Def}\label{def:bounded-absolute-case} 
Let $X$ be an analytic space and let $\gE$ be a set of isomorphism classes of coherent sheaves on  $X$. We say that the set $\gE$ is {\em bounded} if there exist a germ $\tilde K$ of an analytic space around a compact set $K$ and a coherent sheaf $\tilde E$ on $X\times\tilde K$ such that $\gE$ is contained in the set of isomorphism classes of fibers of $\tilde E$ over points of $K$, or, in other words, if there exist an analytic space $S'$, a compact subset $K\subset S'$ and a coherent sheaf $E$ on $X\times S'$ such that $\gE$ is contained in the set of isomorphism classes of fibers of $E$ over points of $K$.
\end{Def}
For schemes defined over $\C$ this definition recovers Grothendieck's definition, \cite[Remark 3.3]{TomaLimitareaI}. 

In the relative case it will be convenient to give a definition of boudedness which is less restrictive than the one proposed in \cite{TomaLimitareaI}. 

\begin{Def}\label{def:bounded-relative-case1}
Let $X$ be an analytic space 
over an analytic space $S$ and $\gE$ a set of isomorphism classes of coherent sheaves on the fibres $X_s$ of $X\to S$. We say that the set $\gE$ is {\em (relatively) bounded (over $S$)} if there exist an at most countable disjoint union $\tilde K:=\coprod\tilde K_i$ of germs of analytic spaces around compact sets over $S$ with $\coprod K_i$ proper over $S$ and a coherent sheaf $\tilde E$ on $X_{\tilde K}$ such that $\gE$ is contained in the set of isomorphism classes of fibers of $\tilde E$ over points of $\coprod K_i$. 
\end{Def}

{\bf Warning:} If $X$ and $S$ are complex spaces and $E$ is a coherent sheaf on $X\times S$ then the isomorphism classes of its fibres $E_s$ over $S$ form a relatively bounded set over $S$ by our Definition \ref{def:bounded-relative-case1}, but the same classes when seen on $X$ in the absolute setting may form an unbounded set according to Definition \ref{def:bounded-absolute-case}. An example is provided by the family given by $(\cO_{\P^1}(n))_{n\in\Z}$ over $\P^1\times \Z$. Another example, this time over an irreducible base $S$, is the Picard line bundle over $X\times \C^*$, where $X$ is an Inoue surface; in this case $\Pic(X)\cong\C^*$. 

The following Remark gives a way of rephrasing Definition \ref{def:bounded-relative-case1}.
\begin{Rem}\label{rem:equivalent-definition} 
A set $\gE$ of isomorphism classes of coherent sheaves on the fibres of $X\to S$  is  (relatively) bounded (over $S$) if and only if for each compact subset $L$ of $S$ there exist a germ $\tilde K$ of analytic space around around a  compact set $K$ over $S$  and a coherent sheaf $\tilde E$ on $X_{\tilde K}$ such that the elements of $\gE$ in each fiber  $X_s$ of $X\to S$ with $s\in L$ are  contained in the set of isomorphism classes of fibers of $\tilde E$ over points of $\tilde K$ covering $s$. \end{Rem}
 
In the above definition it is clear that $\gE$ can be viewed as set of sheaves defined on (some of) the fibers of 
$X_{\tilde K}\to\tilde K$. (We will be loose on the above terminology and often say ``sheaves'' instead of ``isomorphism classes of coherent sheaves''.)

As in \cite{TomaLimitareaI} one can prove the following basic properties of bounded sets of sheaves. We will not reproduce the similar proof here

\begin{Prop}\label{prop:limitari-diverse} 
Let $X$ be an analytic space
 over an analytic space $S$ and let $\gE$, $\gE'$ be two bounded sets of isomorphism classes of sheaves on the fibers of $X\to S$. 
 Suppose that there exists a closed subset $L$ of $X$ proper over $S$ such that the supports of the sheaves from $\gE$ or from $\gE'$ are contained in $L$.
 Then the following sets are bounded as well:
\begin{enumerate}
 \item The sets of kernels, cokernels and images of sheaf homomorphisms $F\to F'$, when the isomorphism
 classes of $F$ and $F'$ belong to $\gE$ and $\gE'$ respectively.
\item The set of isomorphism classes of extensions of $F$ by $F'$, for
  $F$ and $F'$ as above.
\item The set of isomorphism classes of tensor products $F\otimes F'$,  for
  $F$ and $F'$ again as above.
\end{enumerate}
\end{Prop}

Our main result is the following boundedness criterion.

\begin{Th}\label{theorem:principala}
Let $S$ be an analytic space and $X$ an analytic space proper over $S$. Suppose that $X/S$ is endowed with a  relative degree system $(\deg_{d,s},...,\deg_{0,s})_{s\in S}$ coming from  differential forms. 
Let $\gF$ be a set of isomorphism classes of coherent sheaves of dimension at most $d$ on the fibres of $X/S$. For each $s\in S$ we denote by $\gF_s$ the subset of $\gF$ containing the sheaves on the fibre $X_s$ of $X/S$.
 Then the set $\gF$ is bounded if and and only if the following conditions are fulfilled:
\begin{enumerate}
\item There exists a bounded set $\gE$ of isomorphism classes of coherent sheaves on the fibres of $X/S$
 such that each element of $\gF$ is a quotient of an element of $\gE$.
\item There exists a system $\delta=(\delta_d,...,\delta_0)$ of continuous functions on $S$ such that for all $s\in S$, for all $j\in\{0,...,d\}$ and for all $F\in\gF_s$ one has
$$\deg_{j,s} (F) \le\delta_j(s).$$
\end{enumerate}
\end{Th}

In general we shall call a set $\gF$ of isomorphism classes of coherent sheaves on the fibers of $X\to S$ {\em dominated} if it satisfies the first condition of the criterion.

The "only if" is quickly dealt with as follows. The first assertion is clear. For the second one  choose  an at most countable disjoint union $\tilde K:=\coprod_{i\in \N}\tilde K_{i}$ of germs of analytic spaces around compact sets over $S$ with $\coprod K_i$ proper over $S$ and a coherent sheaf $\tilde F$ on $X_{\tilde K}$ such that $\gF$ is contained in the set of isomorphism classes of fibers of $\tilde F$ over points of $\coprod K_i$ as in Definition \ref{def:bounded-relative-case1}. By noetherian induction and flattening we reduce ourselves to the case where  $\tilde F$ is flat over $\tilde K$. The degree functions $\deg_{j}$ of this sheaf will be  continuous on  $\tilde K$ and therefore attain there maxima $m_{j,n}$ on $\cup_{i=0}^nK_i$. Take now an exhaustive sequence $(L_k)_{k\in\N}$ of compact subsets of $S$, i.e. such that $\cup_kL_k=S$ and $L_k\subset \cringle{L}_{k+1}$, $\forall k\in\N$. Then there exists an increasing sequence $(n_k)_{k\in\N}$ of positive integers such that the image of $K_i$ in $S$ does not meet $L_k$ as soon as $i\ge n_k$. It suffices then to construct continuous functions $\delta_j$ on $S$ such that $\delta_j|_{L_k}\ge  m_{j,n_k}$ for all $k\in\N$.

The proof of the "if" part follows the same strategy as in  \cite{Grothendieck-theHilbertScheme} or in \cite{TomaLimitareaI} to use devissages of arbitrary coherent sheaves into pure sheaves and noetherian induction to reduce the difficulty to a simpler boundedness statement on pure sheaves which works under assumptions of domination and bound of degrees in only one dimension. This will be the content of Lemma \ref{tehnica}.  We will not rewrite this part of the proof since it works in our context exactly as in 
\cite{TomaLimitareaI}.  It is when dealing with pure sheaves in Lemma \ref{tehnica} that the difference between the projective case and the K\"ahler case appears. Whereas Grothendieck uses projections on linear subspaces in \cite{Grothendieck-theHilbertScheme}, we reduce the problem to the control of volumes of certain analytic cycles in appropriate cycle spaces and conclude by Bishop's theorem. In \cite{TomaLimitareaI} we worked out this volume control globally by dealing with the  cohomolgy classes of the metrics and the cycles involved.  Here we will use the local volume estimates of Section \ref{section:volume} which allow us to deal with the more general set-up we are considering.
Recall that by Bishop's theorem the subset of the relative analytic cycle space $C_{X/S,d}$ parameterizing cycles in the fibers of $X/S$ whose volumes are bounded by some function $\delta:S\to \R$ is proper over $S$, cf. \cite{BarletMagnusson}.

The same argument proves the following more general result for which we introduce more notation: we denote by $N_r(F)$  the sheaf of sections of a coherent sheaf $F$ whose support have dimension less than $r$ and set $F_{(r)}:=F/N_r$. (If $F$ is of dimension $d$ then
$F_{(d)}$ is pure of dimension $d$.)

\begin{Th}\label{thm:version2}
Suppose that $X/S$ is a proper relative analytic space  endowed with a relative degree system $(\deg_{d,s},...,\deg_{d',s})_{s\in S}$ coming from  differential forms with $d>d'$, 
$\gF$ is a dominated set of isomorphism classes of coherent sheaves of dimension at most $d$ on the fibres of $X/S$ and  $\delta=(\delta_d,...,\delta_{d'})$ is a system of continuous functions on $S$ such that for all $s\in S$, for all $j\in\{d',...,d\}$ and for all $F\in\gF_s$ one has
$$\deg_{j,s} (F) \le\delta_j(s).$$
Then 
\begin{enumerate}
\item
the set of isomorphism classes of sheaves $\{ F_{(d'+1)} \ | \ F\in \gF\}$ on the fibers of $X/S$ is bounded and
\item
there exists a system of continuous functions
$\delta'=(\delta'_d,...,\delta'_{d'})$ such that for all $s\in S$, for all $j\in\{d',...,d\}$ and for all $F\in\gF_s$ one has
$$\deg_{j,s} (F) \ge\delta'_j(s).$$
\end{enumerate} 
\end{Th}

\begin{Lem} (Key Lemma) \label{tehnica}
 Let $X/S$ be an equidimensional proper relative analytic space of relative dimension $d$ with irreducible general fibers  and let $(\deg_{d,s},\deg_{d-1,s})_{s\in S}$ be a relative degree system coming from  differential forms. Let $\delta_{d-1}:S\to\R$ be some continuous function. 
If $\gF$ is a dominated set of classes of pure $d$-dimensional sheaves on the 
fibers  of $X/S$ such that for all $s\in S$ and for all $F\in\gF_s$ one has $\deg_{d-1,s} (F) \le\delta(s)$, then $\gF$ is bounded.
\end{Lem}

The proof follows exactly the same strategy as the proof of \cite[Lemma 4.3]{TomaLimitareaI} but replaces the global volume estimates by the local volume estimates performed in Section \ref{section:volume}. We recall its main line. 

Since $\gF$ is dominated we may suppose, after some base change possibly, that a coherent sheaf $G$ on $X$ exists such that all sheaves $F$ in $\gF$ may be realised as quotients of $G$. 
By using a flattening stratification for $G$ this implies a bound on $\deg_{d,s} F$ by some continuous function $\delta_{d}:S\to\R$ on $S$. 
The statement to prove is local on $S$ and locally we will have a bound on the generic ranks of the sheaves $F\in\gF$. 
We may suppose that this generic rank is constant equal to $r$, say. Then we produce the subspaces $P_2$ of $\P(\bigwedge ^rG)$ of bounded pseudo-volume as in Section \ref{section:quotient-sheaves}. The estimates of  Section \ref{section:volume} imply then a bound on a true volume function on these subspaces. Hence by Bishop's Theorem they represent points in a subset in the relative cycle space $C_{\P(\bigwedge ^rG)/S}$ which is proper over $S$. We go by base change to this relative cycle space and we produce an analytic family of sheaves of the type $F_1$ described in Section \ref{section:quotient-sheaves} by pushing forward the restriction of the sheaf $\cO_{\P(\bigwedge ^rG)}(1)$.
The final step of Section \ref{section:quotient-sheaves} through Lemma \ref{lem:determinant_trick} tells us now how we can recover the sheaves $F\in \gF$ from this family. 


\section{Corollaries}\label{sectioncor}

In this section we present some consequences of Theorem \ref{theorem:principala}.

\subsection{Properness of the Douady space}

We give here a number of applications of our criterion to the properness of irreducible or connected components of the Douady space  partially extending results of Fujiki in \cite{FujikiClosedness}, 
\cite{Fujiki-DouadySpaceII}.
 
\begin{Cor}\label{cor:Douady}
Let $X$ be a proper analytic space over an analytic space $S$, endowed with a  relative degree system $(\deg_{d,s},...,\deg_{0,s})_{s\in S}$ coming from  differential forms and let $E$ be a coherent sheaf on $X$. 
Further let  $\delta=(\delta_d,...,\delta_0)$ be a system of continuous functions on $S$ and consider the subset $D_{E/X/S,d,\le\delta}\subset D_{E/X/S,d}$ of the relative Douady space $D_{E/X/S,d}$ of quotients of $E$ of dimension at most $d$ given by those quotients with degrees bounded by the functions $\delta_j$. 
Then $D_{E/X/S,d,\le\delta}$ is closed in $D_{E/X/S,d}$ with respect to the standard topology and proper over $S$. If moreover the  functions $\delta_j$ are locally constant on $S$ as well as the degree functions on flat families (which means that the degree system is strong), then $D_{E/X/S,d,\le\delta}$ is an analytic subset of $D_{E/X/S,d}$, especially in this case  the connected components of  the Douady space $D_{E/X/S,d}$ are proper over $S$.
\end{Cor}
\begin{pf}
The proof is a straightforward generalization of the absolute case as presented in \cite[Corollary 5.3]{TomaLimitareaI}. For the convenience of the reader we present its line. 
The fact that $D_{E/X/S,d,\le\delta}$ is closed in $D_{E/X/S,d}$ follows from the continuity of the degree functions in flat families and the final statement is easy.
We thus only need to check the fact that $D_{E/X/S,d,\le\delta}$ is proper over $S$. 

By Theorem \ref{theorem:principala} the set $\gF$ of quotients of $E$ parametrized by $D_{E/X/S,d,\le\delta}$ is bounded over $S$. Thus there exist  an at most countable disjoint union $\tilde K:=\coprod\tilde K_i$ of germs an analytic spaces around compact sets over $S$ with $\coprod K_i$ proper over $S$ and a coherent sheaf $\tilde F$ on $X_{\tilde K}$ such that $\gF$ is contained in the set of isomorphism classes of fibres of $\tilde F$ over points of $\coprod K_i$. Consider now the contravariant functor which to a complex space $T$ over $\tilde K$ associates the set $\Hom_{X_T}(E_T,\tilde F_T)$. 
By \cite[Section 3.2]{FlennerEXT} this functor is represented by a linear space $V$ over $\tilde K$. 
It is easy to see that for each $i$ there exists a compact subset $K'_i$ of $V$ covering $K_i$ such that up to a multiplicative constant each non-zero morphism in $\Hom(E_t,F_t)$, $t\in K_i$ is represented by some element of $K'_i$, see the construction of the projective variety over $\tilde K$ associated to $V$ in \cite[Section 1.9]{Fischer}. 
The image $\tilde F'$ of the universal morphism over $\tilde K'$ is a quotient of $E_{\tilde K'}$. By Hironaka's flattening theorem and noetherian induction we may assume that $\tilde F'$ is flat over $\tilde K'$ and that the set of its fibres over $\coprod K'_i$  contains $\gF$. Thus the image of $\coprod K'_i$ through the morphism $\phi:\tilde K'\to D_{E/X/S}$  provided by the universal property of the relative Douady space contains $D_{E/X/S,d,\le\delta}$. But $\coprod K'_i$ is proper over $S$, hence its image $\phi(\coprod K'_i)$ through $\phi$ in $D_{E/X/S}$ is also proper over $S$. As $D_{E/X/S,d,\le\delta}$ is closed in $\phi(\coprod K'_i)$ it follows that $D_{E/X/S,d,\le\delta}$ is proper over $S$ as well.  
\end{pf}

Our next concern is to say something about irreducible components of the relative Douady space of quotients containing points representing pure quotient sheaves. We shall restrict ourselves here to the case of pure sheaves and of degrees in dimensions $d$ and $d-1$ but analogous statements for more general degree systems may be proved using Theorem \ref{thm:version2}.
We start with a remark which works in a very general context.
\begin{Rem}\label{rem:openness-of-purity} 
If $X$ is an analytic space proper over $S$ and if $E$ is a flat family of $d$-dimensional coherent sheaves on the fibres of $X/S$, then the set of points $s\in S$ such that $E_s$ is not pure is a closed analytic subset of $S$. In other words, purity is a Zariski open property in flat families of coherent sheaves.
\end{Rem} This may be proved by adapting  Maruyama's approach in  \cite[Proposition 1.13]{Maruyama-Construction} to our context. It suffices to work in loc. cit. with local resolutions and apply the purity criterion from \cite[Lemma 1.12]{Maruyama-Construction}.

\begin{Cor}\label{cor:Douady-pure}
Let $X$ be a proper analytic space over an analytic space $S$, endowed with a strong relative degree system $(\deg_{d,s},\deg_{d-1,s})_{s\in S}$ coming from  differential forms and let $E$ be a coherent sheaf on $X$. 
Further let  $m$ be a real number and consider the union 
$D'_{E/X/S,=d,\le m}$
of irreducible components of  the relative Douady space $D_{E/X/S,d}$ of quotients of $E$ of dimension $d$ with $(d-1)$-dimensional degrees bounded by $m$ and such that each such irreducible component contains a point represented by a pure quotient sheaf. 
Then $D'_{E/X/S,=d,\le m}$ is proper over $S$.
\end{Cor}
\begin{pf}
By Remark \ref{rem:openness-of-purity} the subset parameterizing pure quotient sheaves will be Zariski open and (classically) dense in each irreducible component of $D'_{E/X/S,=d,\le m}$. We follow the constructions and the arguments of the proof of Corollary \ref{cor:Douady-pure}  applying Theorem \ref{thm:version2} instead of Theorem \ref{theorem:principala} and with the same notations. It can be seen that it applies to our new situation 
with the single difference that we a priori only know that the image of $\coprod K'_i$ through the morphism $\phi:\tilde K'\to D_{E/X/S}$  provided by the universal property of the relative Douady space contains the  subset $D'_{E/X/S,=d,\le m,pure}$ of $D'_{E/X/S,=d,\le m}$ parameterizing pure quotient sheaves. But this subset is dense in $D'_{E/X/S,=d,\le m}$. This and the fact that $\coprod K'_i$ is proper over $S$ immediately imply that $\phi(\coprod K'_i)$ contains  $D'_{E/X/S,=d,\le m}$. But $D'_{E/X/S,=d,\le m}$ is closed in $D_{E/X/S,d}$ hence also in  $\phi(\coprod K'_i)$ and is thus proper over $S$.  
\end{pf}

As easy consequences we deduce the following versions with parameters.

\begin{Cor}\label{cor:Douady-cu-parametru}
Let $X$ be a proper analytic space over an analytic space $S$, let $T$ be a locally compact separated topological space and let $(\deg_{d,(s,t)},...,\deg_{0,(s,t)})_{(s,t)\in S\times T}$ be a continuous family of relative degree systems on $X/S$ parameterized by $S\times T$ and coming from  differential forms. Let $E$ be a coherent sheaf on $X$. 
Further let  $\delta=(\delta_d,...,\delta_0)$ be a system of continuous functions on $S\times T$ and consider the subset $D_{E/X/S,T,d,\le\delta}\subset D_{E/X/S,d}\times T$ of the product of the relative Douady space $D_{E/X/S,d}$ by $T$ of pairs of quotients of $E$ of dimension at most $d$ and points of $T$, given by those quotients with degrees bounded by the functions $\delta_j$. 
Then $D_{E/X/S,T, d,\le\delta}$ is closed in $D_{E/X/S,d}\times T$  and proper over $S\times T$. \end{Cor}
\begin{pf} By the continuity of the degree functions it is clear that $D_{E/X/S,T,d,\le\delta}$ is closed in $D_{E/X/S,d}\times T$. The fact that it is proper over $D_{E/X/S,d}\times T$  is a local property over $S\times T$, so we can look at a product of compact subsets $L$ of $S$ and $K$ of $T$. For $d\ge j\ge 0$ we define $m_j$ to be the maximum of the function $\delta_j$ on $L\times K$ and put $m=(m_d,...,m_0)$. Then over $L\times K$ we have   $D_{E/X/S,T,d,\le\delta}\subset D_{E/X/S,d,\le m}\times T$ and this space is proper over $S\times T$ by Corollary \ref{cor:Douady}.
  \end{pf}

\begin{Cor}\label{cor:Douady-cu-parametru-pure}
Let $X$ be a proper analytic space over an analytic space $S$, let $T$ be a locally compact separated topological space and let $(\deg_{d,(s,t)},\deg_{d-1,(s,t)})_{(s,t)\in S\times T}$ be a continuous family of strong relative degree systems on $X/S$ parameterized by $S\times T$ and coming from  differential forms. 
Let $E$ be a coherent sheaf on $X$. 
Further let  $\delta$ be a  continuous function on $ T$ and consider the subset 
$D'_{E/X/S,T,=d,\le\delta}:=\cup_{t\in T}(D'_{E/X/S,=d,\le\delta(t)}\times \{t\}) $ of $D_{E/X/S,d}\times T$, where, as in Corollary \ref{cor:Douady-pure}, $D'_{E/X/S,=d,\le\delta(t)}$ denotes the union of irreducible components of the relative Douady space $D_{E/X/S,d}$ of quotients of $E$ of dimension $d$ with $(d-1)$-dimensional degrees bounded by $\delta(t)$ and such that each such irreducible component contains a point represented by a pure quotient sheaf. 
Then $D'_{E/X/S,T,=d,\le\delta}$ is closed in $D_{E/X/S,d}\times T$  and proper over $S\times T$. 
\end{Cor}
\begin{pf} As before the properties we have to check on $D'_{E/X/S,T,=d,\le\delta}$ are local over $S\times T$ so we look at what happens over a a product $K\times L$ of compact subsets $L$ of $S$ and $K$ of $T$. Set $m$ to be the maximum of the function $\delta$ on $ K$. 
Then over $S\times K$ we have $D'_{E/X/S,K,d,\le\delta}= (D'_{E/X/S,=d,\le m}\times K)\cap D_{E/X/S, K, =d,\le\delta}$, where  $D_{E/X/S,K, d,\le\delta}:=\cup_{t\in K}\{([q:E_s\to F],t)\in(D_{E/X/S,d}\times \{t\}) \ | \ s\in S, \ \deg_{d-1,(s,t)}(F)\le \delta(t)\}$ is defined in a similar way to the space of Corollary \ref{cor:Douady-cu-parametru}. Now the space $D'_{E/X/S,=d,\le m}\times K$ is proper over $S\times K$ by Corollary  \ref{cor:Douady-pure} and $D_{E/X/S,K, d,\le\delta}$ is closed in $D_{E/X/S,d}$ by Corollary \ref{cor:Douady-cu-parametru} whence the desired conclusion.
  \end{pf}

\subsection{Semistable sheaves}


The existence of an absolute Harder-Narasimhan filtration is known for a variety of cases where a suitable stability notion has been introduced; see \cite{Andre} for a systematic treatment. We will be concerned here with the existence of a relative Harder-Narasimhan filtration in the following set-up, cf. \cite[Section 2.3]{HL} for the projective algebraic case. 

Let $X$ be an analytic space proper over an irreducible analytic space $S$. Suppose that $X/S$ is endowed with a strong relative degree system $(\deg_{d,s},...,\deg_{d',s})_{s\in S}$ coming from  differential forms.
 Then we may define a {\em slope vector} $\mu(F)$ for any $d$-dimensional coherent sheaf $F$ on some fibre $X_s$ of $X/S$ in the following way, cf. \cite{TomaCriteria}:
$$\mu(F):=(\frac{\deg_{d}(F)}{\deg_d(F)},\frac{\deg_{d-1}(F)}{\deg_d(F)},..., \frac{\deg_{d'}(F)}{\deg_d(F)})\in \R^{d-d'+1}.$$ 
We shall call such a sheaf $F$  {\em semistable} if it is pure and if for any non-trivial coherent subsheaf $F'\subset F$ we have $\mu(F')\le\mu(F)$ with respect to the lexicographic order. (As usual we will say that $F$ is {\em stable} if the above inequality between slope vectors is strict for all proper subsheaves $F'$ of $F$). As in \cite[Section 1.3]{HL} one checks that any pure $d$-dimensional sheaf $F$ on some fibre $X_s$ of $X/S$ admits a   Harder-Narasimhan filtration i.e. a unique increasing filtration
$$0=HN_0(F)\subset  HN_1(F)\subset...\subset HN_l(F)=F,$$
with semistable factors $HN_i(F)/HN_{i-1}(F)$ for $1\le i\le l$ and such that
$$\mu( HN_1(F)/HN_{0}(F))>...>\mu( HN_l(F)/HN_{l-1}(F)).$$
We will write $\mu_{\min}(F):=\mu( HN_l(F)/HN_{l-1}(F))$ and $\mu_{\max}(F):=\mu( HN_1(F)/HN_{0}(F))$. It is easy to see that $F$ is semistable if and only if $\mu_{\min}(F)=\mu(F)$ if and only if $\mu_{\max}(F)=\mu(F)$.

By a
{\em relative Harder-Narasimhan filtration} for a flat family $E$ of $d$-dimensional coherent sheaves on the fibres of $X/S$ we mean a proper bimeromorphic morphism of irreducible analytic spaces $T\to S$ and a filtration 
$$0=HN_0(E)\subset  HN_1(E)\subset...\subset HN_l(E)=E_T$$
such that the factors $HN_i(E)/HN_{i-1}(E)$ are flat over $T$ for $1\le i\le l$ and which induces the absolute Harder-Narasimhan filtrations fibrewise over some dense Zariski open subset of $S$.

\begin{Cor}\label{cor:filtrarea HN relativa} (Existence of a relative Harder-Narasimhan filtration).
In the above set-up any flat family $E$ of $d$-dimensional coherent sheaves on the fibres of $X/S$ with pure general members has a  relative Harder-Narasimhan filtration $HN_\bullet(E)$. 
Moreover this filtration has the following universal property:  if $f:T'\to S$ is a bimeromorphic morphism of irreducible analytic spaces and if $F_\bullet$ is a filtration of $E_{T'}$ with flat factors, which coincides  fibrewise with the absolute Harder-Narasimhan filtration over general points $s\in S$, then $f$ factorizes over $T$ and $F_\bullet=HN_\bullet(E)_{T'}$.
\end{Cor}
\begin{pf}
We start by working over a compact neighbourhood $L$ of a fixed point $s_0$ of $S$.

The support cycles of the flat family $E$ over $S$ give rise to a section of the relative Barlet cycle space $C_{X/S,d}$ over $S$, \cite{Rydh}, \cite{Fog69}, \cite{BarletMagnusson}.
Using \cite[Proposition IV.7.1.2]{BarletMagnusson} it then follows that the set $\{\deg_d(C) \ | \ C\in (C_{X/S,d})_s, \ 0< C\le \Supp(E_s), \ s\in L\}$ is finite. 

Let $r$ denote its minimum. For a $d$-dimensional quotient $F$ of some $E_s$ with $s\in L$ we get the implication 
$$\frac{\deg_{d-1}(F)}{\deg_d(F)}\le\frac{\deg_{d-1}(E_s)}{\deg_d(E_s)} \Rightarrow   \deg_{d-1}(F)\le \frac{r\deg_{d-1}(E_s)}{\deg_d(E_s)}=:m.$$
By Corollary \ref{cor:Douady-pure}   the union 
$D:=D'_{E/X/S,=d,\le m}$
of irreducible components of  the relative Douady space $D_{E/X/S,d}$ of quotients of $E$ of dimension $d$ with $(d-1)$-dimensional degrees bounded by $m$ and such that each such irreducible component contains a point represented by a pure quotient sheaf is proper over $S$. 
If the map $D_{\mathring L}\to \mathring L$ is not surjective then one can see easily that the set of points $s\in S$ such that $E_s$ is not semistable is a closed analytic proper subset of $S$. It follows in particular that semistability is a Zariski open property in flat families of coherent sheaves. 

Suppose now that the map $D_{\mathring L}\to \mathring L$ is surjective. One consequence of the above argument is that $D_{\mathring L}$ has only finitely many irreducible components. Let $D'$ denote the union of those irreducible components  of $D$ covering $L$. We know that the slope vectors of quotients parametrized by such a component are constant along the component. In particular we may choose the minimal such slope vector $\mu_0$ among those which appear on irreducible components of $D'$. Using the properness of $D'_{\mathring L}\to \mathring L$ and the properties of the absolute Harder-Narasimhan filtration we can see that there is a unique irreducible component $D_{\min}$ of $D'$ which realizes $\mu_0$ as slope vector of the quotients it parametrizes and such that over general points $s$ of $L$ the corresponding quotients correspond to the last quotient of $HN_\bullet(E_s)$. It follows that $D_{\min}\to S$ is a proper modification. Let then $F$ denote the universal quotient of $E_{D_{\min}}$ and $G:=\Ker(E_{D_{\min}}\to F)$ the universal kernel. The family $G$ is a flat family of $d$-dimensional coherent sheaves on the fibres of $X_{D_{\min}}/{D_{\min}}$ with pure general members. We apply the same procedure to this new family. The process ends after a finite number of steps and gives the desired  relative  Harder-Narasimhan filtration of $E$ over $S$.

Using the above construction and the properties of the Douady space and of the absolute Harder-Narasimhan filtration one easily obtains the claimed universal property for the relative  Harder-Narasimhan filtration.
\end{pf}

One consequence of the above proof is the openness of semistability  in flat families of coherent sheaves:

\begin{Cor}\label{cor:deschiderea-semistabilitatii}
Let $X/S$ be a proper flat family of $d$-dimensional complex spaces endowed with a strong relative degree system $(\deg_{d,s},...,\deg_{d',s})_{s\in S}$ coming from  differential forms and let $E$ be a coherent sheaf on $X$ flat over $S$ of relative dimension $d$. Then the set of points of $S$ over which the fibers of $E$ are semistable form a Zariski open subset of $S$. 
\end{Cor}

\begin{Rem}\label{rem:strongness} 
The hypothesis of strongness in Corollary \ref{cor:deschiderea-semistabilitatii} cannot be relaxed as can be seen in the case when $X$ is a non-k\"ahlerian compact complex surface. As remarked in Section \ref{subsection:degree-systems}  a Gauduchon metric on the surface gives a complete degree system on $X$. However in general semistability with respect to such a degree system is not an open condition in flat familes, see \cite[Example 2.1]{BuTeTo1bis}. Note however that Theorem \ref{theorem:principala} holds without a strongness hypothesis.
\end{Rem} 

The next Corollary deals with openness of stability in a different context and is another example of extending to non necessarilly projective manifolds statements on semistable sheaves which were previously proved only under projectivity assumptions, cf. \cite[Theorem 3.3]{GKP-Movable}. It is an easy consequence of Corollary \ref{cor:Douady-cu-parametru-pure}.

\begin{Cor}\label{cor:intrebareMihai}
Let $X$ be a compact complex space of pure dimension $d$ and let $(\Omega_t)_{t\in T}$ be    a continuous family of strictly positive $\d$-closed forms defining $(d-1)$-degree functions on $X$ and parameterized by a locally compact separated topological space $T$. Suppose morever that $X$ is also endowed with a volume form allowing to compute $d$-degrees and that $F$ is a $d$-dimensional pure sheaf on $X$ which is stable with respect  to $\Omega_{t_0}$ for some $t_0\in T$. Then there exists a neighbourhood $V$ of $t_0$ in $T$ such that $F$ is $\Omega_t$-stable for all $t\in V$. 
\end{Cor}

In the particular (semi-K\"ahlerian) case when $X$ is a compact complex manifold and $T$ is an open subset of $H^{d-1,d-1}(X,\R)$ one can give an alternative proof using only the absolute case of Corollary \ref{cor:Douady-pure} and an argument as in \cite[Lemma 6.7]{GRT1}. One can in this case choose a convex polytope as compact neighbourhood $L$ and note that the degree functions evaluated on a fixed coherent sheaf $E$ are affine linear.  Hence for any torsion free quotient $E$ of $F$ the function $f_{F,E}(t):=\mu_t(F)-\mu_t(E)$ is affine linear and therefore attains its maximum on $L$ at some vertex $v$ of $L$. Now Corollary  \ref{cor:Douady-pure} says that for each vertex $v$ only finitely many such functions exist which are non-negative at $v$, if any. Then the subset of $L$ where all these functions are non-positive is a convex polytope containing $t_0$ in its interior. This method of proof allows the following slight extension to pseudo-stable sheaves, where "pseudo" means here that the polarizing form with respect to which one considers stability is positive but not necessarilly strictly positive. 

\begin{Cor}\label{cor:intrebareMihai2}
Let $X$ be a compact complex space of pure dimension $d$ endowed with a volume form, let $\Omega_0$, $\Omega_1$ be positive $\d$-closed forms of bi-degree $(d-1,d-1)$ with $\Omega_1$ strictly positive and let  $\Omega_t:=(1-t)\Omega_0+t\Omega_1$ for $t\in[0,1]$.  Suppose morever that $F$ is a $d$-dimensional pure sheaf on $X$ which is pseudo-stable with respect  to $\Omega_{0}$. Then there exists a neighbourhood $V$ of $0$ in $[0,1]$ such that $F$ is $\Omega_t$-stable for all $t\in V\setminus\{0\}$. 
\end{Cor}
\begin{pf}
In the particular case at hand for any coherent sheaf $E$ on $X$ the function $t\mapsto \deg_{\Omega_t}E$ is affine on $[0,1]$ and the same holds for slope functions. By looking at slopes of coherent subsheaves $F'$ of $F$ it follows that if $F$ is both pseudo-stable with respect to $\Omega_0$ and stable with respect to $\Omega_1$ then it is stable with respect to all $\Omega_t$ for $t\in]0,1]$. 

We are left with the case when $F$ is not stable with respect to $\Omega_1$. Since the stability condition may be translated into an inequalities for slopes of pure quotients of $F$ we may apply Corollary \ref{cor:Douady-pure} in the absolute case to conclude that there is a finite number of irreducible components of the Douady space $D_{F/X}$ containing destabilizing  quotients for $F$ with respect to $\Omega_1$, i.e. pure quotients $E$ with $\mu_{\Omega_1}(E)\le\mu_{\Omega_1}(F)$. 
Let $E_1$, ..., $E_k$ be a choice of such quotient sheaves, one for each concerned irreducible component of $D_{F/X}$.  Then the functions $t\mapsto(\mu_{\Omega_t}(E_j)-\mu_{\Omega_t}(F))$ being affine and positive at $0$, they will remain positive on a neighbourhood $V$ of $0$ in $[0,1]$. It is now immediate to check that $F$ will be $\Omega_t$-stable for all $t\in V\setminus\{0\}$. 
\end{pf}


\begin{thebibliography}{BFM79}

\bibitem[AG06]{AG}
Vincenzo Ancona and Bernard Gaveau, \emph{Differential forms on singular
  varieties}, Pure and Applied Mathematics (Boca Raton), vol. 273, Chapman \&
  Hall/CRC, Boca Raton, FL, 2006, De Rham and Hodge theory simplified.

\bibitem[And09]{Andre}
Yves Andr\'e, \emph{Slope filtrations}, Confluentes Math. \textbf{1} (2009),
  no.~1, 1--85.

\bibitem[Bar78]{BarletConvexitate}
Daniel Barlet, \emph{Convexit\'e de l'espace des cycles}, Bull. Soc. Math.
  France \textbf{106} (1978), no.~4, 373--397.

\bibitem[BFM75]{BFM75}
Paul Baum, William Fulton, and Robert MacPherson, \emph{Riemann-{R}och for
  singular varieties}, Inst. Hautes \'Etudes Sci. Publ. Math. (1975), no.~45,
  101--145.

\bibitem[BFM79]{BFM79}
\bysame, \emph{Riemann-{R}och and topological {$K$}\ theory for singular
  varieties}, Acta Math. \textbf{143} (1979), no.~3-4, 155--192.

\bibitem[BH69]{BlHe69}
Thomas Bloom and Miguel Herrera, \emph{De {R}ham cohomology of an analytic
  space}, Invent. Math. \textbf{7} (1969), 275--296.

\bibitem[Bin83]{Bin83}
J\"urgen Bingener, \emph{On deformations of {K}\"ahler spaces. {I}}, Math. Z.
  \textbf{182} (1983), no.~4, 505--535.

\bibitem[BM14]{BarletMagnusson}
Daniel Barlet and J\'on Magn\'usson, \emph{Cycles analytiques complexes. {I}.
  {T}h\'eor\`emes de pr\'eparation des cycles}, Cours Sp\'ecialis\'es
  [Specialized Courses], vol.~22, Soci\'et\'e Math\'ematique de France, Paris,
  2014.

\bibitem[Bre97]{Bredon}
Glen~E. Bredon, \emph{Topology and geometry}, Graduate Texts in Mathematics,
  vol. 139, Springer-Verlag, New York, 1997, Corrected third printing of the
  1993 original.

\bibitem[BS77]{BS}
C.~B\u{a}nic\u{a} and O.~St\u{a}n\u{a}\c{s}il\u{a}, \emph{M\'ethodes
  alg\'ebriques dans la th\'eorie globale des espaces complexes.},
  Gauthier-Villars, Paris, 1977, Avec une pr\'eface de Henri Cartan,
  Troisi\`eme \'edition, Collection ``Varia Mathematica''.

\bibitem[BTT17]{BuTeTo1bis}
Nicholas Buchdahl, Andrei Teleman, and Matei Toma, \emph{A continuity theorem
  for families of sheaves on complex surfaces}, J. Topol. \textbf{10} (2017),
  no.~4, 995--1028. \MR{3743066}

\bibitem[CJM19]{CampanaCaoPaun}
Fr\'{e}d\'{e}ric Campana, Cao Junyan, and P\u{a}un Mihai, \emph{Subharmonicity
  of direct images and applications}, \texttt{arXiv:1906.11317}, 2019.

\bibitem[CP19]{CampanaPaun}
Fr\'{e}d\'{e}ric Campana and Mihai P\u{a}un, \emph{Foliations with positive
  slopes and birational stability of orbifold cotangent bundles}, Publ. Math.
  Inst. Hautes \'{E}tudes Sci. \textbf{129} (2019), 1--49. \MR{3949026}

\bibitem[Dem93]{DemaillyJDG93}
Jean-Pierre Demailly, \emph{A numerical criterion for very ample line bundles},
  J. Differential Geom. \textbf{37} (1993), no.~2, 323--374. \MR{1205448}

\bibitem[Dem12]{DemaillyBook}
\bysame, \emph{Complex analytic and differential geometry}, OpenContentBook,
  2012.

\bibitem[Dim92]{Dimca-Singularities}
Alexandru Dimca, \emph{Singularities and topology of hypersurfaces},
  Universitext, Springer-Verlag, New York, 1992.

\bibitem[Dim04]{Dimca-Sheaves}
\bysame, \emph{Sheaves in topology}, Universitext, Springer-Verlag, Berlin,
  2004.

\bibitem[DP77]{DoPo77}
Pierre Dolbeault and Jean Poly, \emph{Differential forms with subanalytic
  singularities; integral cohomology; residues}, Several complex variables
  ({P}roc. {S}ympos. {P}ure {M}ath., {V}ol. {XXX}, {W}illiams {C}oll.,
  {W}illiamstown, {M}ass., 1975), {P}art 1, Amer. Math. Soc., Providence, R.I.,
  1977, pp.~255--261.

\bibitem[DV76]{DV}
\emph{S\'{e}minaire de {G}\'{e}om\'{e}trie {A}nalytique}, Soci\'{e}t\'{e}
  Math\'{e}matique de France, Paris, 1976, Tenu \`a l'\'{E}cole Normale
  Sup\'{e}rieure, Paris, 1974--75, Dirig\'{e} par Adrien Douady et Jean-Louis
  Verdier, Ast\'{e}risque, No. 36-37. \MR{0424820}

\bibitem[Fis76]{Fischer}
Gerd Fischer, \emph{Complex analytic geometry}, Lecture Notes in Mathematics,
  vol. 538, Springer-Verlag, Berlin, 1976.

\bibitem[FL85]{FL}
William Fulton and Serge Lang, \emph{Riemann-{R}och algebra}, Grundlehren der
  Mathematischen Wissenschaften [Fundamental Principles of Mathematical
  Sciences], vol. 277, Springer-Verlag, New York, 1985.

\bibitem[Fle82]{FlennerEXT}
Hubert Flenner, \emph{Eine {B}emerkung \"uber relative {${\rm Ext}$}-{G}arben},
  Math. Ann. \textbf{258} (1981/82), no.~2, 175--182.

\bibitem[FM81]{FultonMacPh}
William Fulton and Robert MacPherson, \emph{Categorical framework for the study
  of singular spaces}, Mem. Amer. Math. Soc. \textbf{31} (1981), no.~243.

\bibitem[Fog69]{Fog69}
John Fogarty, \emph{Truncated {H}ilbert functors}, J. Reine Angew. Math.
  \textbf{234} (1969), 65--88.

\bibitem[Fuj78]{FujikiClosedness}
Akira Fujiki, \emph{Closedness of the {D}ouady spaces of compact {K}\"ahler
  spaces}, Publ. Res. Inst. Math. Sci. \textbf{14} (1978), no.~1, 1--52.

\bibitem[Fuj82]{Fujiki-Space-of-Divisors}
\bysame, \emph{Projectivity of the space of divisors on a normal compact
  complex space}, Publ. Res. Inst. Math. Sci. \textbf{18} (1982), no.~3,
  1163--1173.

\bibitem[Fuj84]{Fujiki-DouadySpaceII}
\bysame, \emph{On the {D}ouady space of a compact complex space in the category
  {${\mathcal{C}}$}. {II}}, Publ. Res. Inst. Math. Sci. \textbf{20} (1984),
  no.~3, 461--489.

\bibitem[Ful95]{Fulton}
William Fulton, \emph{Algebraic topology}, Graduate Texts in Mathematics, vol.
  153, Springer-Verlag, New York, 1995, A first course.

\bibitem[GH78]{GriffithsHarris}
Phillip Griffiths and Joseph Harris, \emph{Principles of algebraic geometry},
  Wiley-Interscience [John Wiley \& Sons], New York, 1978, Pure and Applied
  Mathematics. \MR{507725}

\bibitem[GKP16]{GKP-Movable}
Daniel Greb, Stefan Kebekus, and Thomas Peternell, \emph{Movable curves and
  semistable sheaves}, Int. Math. Res. Not. IMRN (2016), no.~2, 536--570.

\bibitem[Gor81]{Gor81}
R.~Mark Goresky, \emph{Whitney stratified chains and cochains}, Trans. Amer.
  Math. Soc. \textbf{267} (1981), no.~1, 175--196.

\bibitem[Gro61]{Grothendieck-theHilbertScheme}
Alexander Grothendieck, \emph{Techniques de construction et th\'eor\`emes
  d'existence en g\'eom\'etrie alg\'ebrique. {IV}. {L}es sch\'emas de
  {H}ilbert}, Exp.\ No.\ 221, 249--276.

\bibitem[GRT16]{GRT1}
Daniel Greb, Julius Ross, and Matei Toma, \emph{Variation of {G}ieseker moduli
  spaces via quiver {GIT}}, Geom. Topol. \textbf{20} (2016), no.~3, 1539--1610.

\bibitem[GT17]{GrebToma}
Daniel Greb and Matei Toma, \emph{Compact moduli spaces for slope-semistable
  sheaves}, Algebr. Geom. \textbf{4} (2017), no.~1, 40--78.

\bibitem[Har77]{Hartshorne}
Robin Hartshorne, \emph{Algebraic geometry}, Graduate Texts in Mathematics,
  vol.~52, Springer-Verlag, New York, 1977.

\bibitem[Her66]{Her66}
Miguel~E. Herrera, \emph{Integration on a semianalytic set}, Bull. Soc. Math.
  France \textbf{94} (1966), 141--180.

\bibitem[HL10]{HL}
Daniel Huybrechts and Manfred Lehn, \emph{The geometry of moduli spaces of
  sheaves}, second ed., Cambridge Mathematical Library, Cambridge University
  Press, Cambridge, 2010.

\bibitem[Ive86]{Iversen}
Birger Iversen, \emph{Cohomology of sheaves}, Universitext, Springer-Verlag,
  Berlin, 1986.

\bibitem[Kob87]{KobayashiVectorBundles}
Shoshichi Kobayashi, \emph{Differential geometry of complex vector bundles},
  Publications of the Mathematical Society of Japan, vol.~15, Princeton
  University Press, Princeton, NJ, 1987.

\bibitem[Lev87]{Levy1987}
Roni~N. Levy, \emph{The {R}iemann-{R}och theorem for complex spaces}, Acta
  Math. \textbf{158} (1987), no.~3-4, 149--188.

\bibitem[Lev08]{Levy-bivariant}
\bysame, \emph{Riemann-{R}och theorem for higher bivariant {$K$}-functors},
  Ann. Inst. Fourier (Grenoble) \textbf{58} (2008), no.~2, 571--601.

\bibitem[LT95]{LuTe}
Martin L{\"u}bke and Andrei Teleman, \emph{The {K}obayashi-{H}itchin
  correspondence}, World Scientific Publishing Co., Inc., River Edge, NJ, 1995.

\bibitem[Mar96]{Maruyama-Construction}
Masaki Maruyama, \emph{Construction of moduli spaces of stable sheaves via
  {S}impson's idea}, Moduli of vector bundles ({S}anda, 1994; {K}yoto, 1994),
  Lecture Notes in Pure and Appl. Math., vol. 179, Dekker, New York, 1996,
  pp.~147--187.

\bibitem[MM07]{MaMarinescu}
Xiaonan Ma and George Marinescu, \emph{Holomorphic {M}orse inequalities and
  {B}ergman kernels}, Progress in Mathematics, vol. 254, Birkh\"{a}user Verlag,
  Basel, 2007. \MR{2339952}

\bibitem[Ryd08]{Rydh}
David Rydh, \emph{Families of cycles and the {C}how scheme}, Ph.D. thesis, KTH
  Stockholm, May 2008, 218 p., electronic version available from the following
  {URL} \texttt{http://www.math.kth.se/$\sim$dary/thesis/thesis.pdf}.

\bibitem[Sto71]{Stoll-fiber-integration}
Wilhelm Stoll, \emph{Fiber integration and some applications}, Symposium on
  {S}everal {C}omplex {V}ariables ({P}ark {C}ity, {U}tah, 1970), Springer,
  Berlin, 1971, pp.~109--120. Lecture Notes in Math., Vol. 184. \MR{0300302}

\bibitem[Tom16]{TomaLimitareaI}
Matei Toma, \emph{Bounded sets of sheaves on {K}\"ahler manifolds}, J. Reine
  Angew. Math. \textbf{710} (2016), 77--93.

\bibitem[Tom20]{TomaCriteria}
\bysame, \emph{Properness criteria for families of coherent analytic sheaves},
  Algebr. Geom. 7, 486-502, 2020.

\bibitem[Var89]{Var89}
Jean Varouchas, \emph{K\"ahler spaces and proper open morphisms}, Math. Ann.
  \textbf{283} (1989), no.~1, 13--52.

\end{thebibliography}

\def\cprime{$'$} \def\polhk#1{\setbox0=\hbox{#1}{\ooalign{\hidewidth
  \lower1.5ex\hbox{`}\hidewidth\crcr\unhbox0}}}
  \def\polhk#1{\setbox0=\hbox{#1}{\ooalign{\hidewidth
  \lower1.5ex\hbox{`}\hidewidth\crcr\unhbox0}}}
\providecommand{\bysame}{\leavevmode\hbox to3em{\hrulefill}\thinspace}
\providecommand{\MR}{\relax\ifhmode\unskip\space\fi MR }
\providecommand{\MRhref}[2]{%
  \href{http://www.ams.org/mathscinet-getitem?mr=#1}{#2}
}
\providecommand{\href}[2]{#2}


\end{document}